\documentclass{article}
\usepackage[title]{appendix}
\usepackage{hyperref}
\usepackage[T1]{fontenc}
\usepackage[utf8]{inputenc}
\usepackage[english]{babel}
\usepackage{a4wide}
\usepackage{amsmath,amssymb}
\usepackage{amsthm}
\usepackage[autolanguage,np]{numprint}
\usepackage{graphicx}
\usepackage{subfig,float}
\usepackage{mathtools}
\usepackage{enumitem}
\usepackage{csquotes}
\usepackage[compact,small]{titlesec}
\usepackage{graphicx}
\usepackage{tikz}
\usetikzlibrary{trees,arrows}
\usepackage{float}
\usepackage{cleveref} 

\newcommand{\X}{\mathbf{X}}
\newcommand{\G}{\mathbf{G}}
\newcommand{\g}{\mathbf{g}}

\newcommand{\cM}{\mathcal{M}}

\newcommand{\cG}{\mathcal{G}}

\newcommand{\R}{\mathbb{R}}
\newcommand{\x}{\mathbf{x}}
\newcommand{\p}{\mathbf{p}}
\newcommand{\HH}{\overline{H}}
\newcommand{\y}{\mathbf{y}}
\newcommand{\0}{{\bf{0}}}
\newcommand{\e}{\textrm{e}}

\def\rp{\right)}
\def\lp{\left(}

\titlelabel{\thetitle. }

\newtheorem{lemma}{Lemma}
\newtheorem{proposition}{Proposition}
\newtheorem{theorem}{Theorem}

\newtheorem{remark}{Remark}

\theoremstyle{definition}
\newtheorem{example}{Example}

\renewcommand{\P}{\mathbb{P}}
\newcommand{\EE}{\mathbb{E}}

\newcommand{\indicator}{\mathbf{1}}

\title{Diffusion spiders: Green kernel, excessive functions \\ and optimal stopping}
\author{Jukka Lempa
\\{\small University of Turku}
\\{\small Department of Mathematics and Statistics}
\\{\small FI-20014 Turun Yliopisto, Finland} \\{\small email: jukka.lempa@utu.fi}
\and
Ernesto Mordecki \\{\small University of the Republic}
\\{\small Faculty of Science} \\ {\small Center of Mathematics}
\\{\small 11400 Montevideo, Uruguay} \\{\small email: mordecki@cmat.edu.uy}
\and
Paavo Salminen\\{\small Abo Akademi University}
\\{\small Faculty of Science and Engineering}
\\{\small FI-20500 Abo, Finland} \\{\small email: phsalmin@abo.fi}
}
\date{}

\begin{document}

\maketitle

\begin{abstract}
\noindent
A diffusion spider is a strong Markov process with continuous paths taking values on a graph with one vertex and a finite number of edges (of infinite length).  An example is Walsh's Brownian spider where the process on each edge behaves as Brownian motion. We calculate  the density of the resolvent kernel in terms of the characteristics of the underlying diffusion. Excessive functions are studied via the Martin boundary theory. The main result is an explicit expression for the representing measure of a given excessive function. These results are used to solve  optimal stopping problems for diffusion spiders.



\end{abstract}

\bigskip






\newpage
\section{Introduction}
\label{II}
Diffusion on graphs has been a topic of intense research in the past years.
One of the earliest papers in this direction is by Walsh \cite{Walsh}, where he introduces a generalization of the skew Brownian motion, later called the Walsh Brownian motion.
This process lives on the half lines in $\R^2$ with only one common point, the origin of the plane. On the rays (or legs) the process behaves like a Brownian motion, and when this Brownian motion hits the origin, it chooses, roughly speaking,  a new ray randomly with an angle having a probability distribution defined on $[0,2\pi)$.
A special case is when the distribution consists of finitely many atoms. If $n$ is the number of atoms then the process is called a Brownian (or Walsh) spider with $n$ legs. The Brownian spider can be seen as a diffusion on a star shaped graph.

Salisbury \cite{Salisbury} gives a formal construction of the Skew Brownian motion (the particular case with two legs) based on the excursion theory of the right processes, see also an earlier paper by Rogers \cite{Rogers}. 
The Walsh Brownian motion was considered by Barlow, Pitman and Yor in \cite{BarlowPitmanYor1}. 
They construct this process using  the semi-group theory. 
Freidlin and Wentzell \cite{FreidlinWentzell} study diffusions on graphs through martingale problems, provide methods to verify the weak convergence of discrete processes  to these diffusions, and
exhibit several examples.
Friedlin and Sheu \cite{fs} prove the  It\^o formula for diffusions on  graphs and establish  large deviation principles.
Papanicolau et al. \cite{PapaPapaLep} compute exit probabilities and related quantities for a Brownian spider with equal leg probabilities. This is  extended in Fitzsimmons and Kuter \cite{fk} for Walsh Brownian motion (on more general graphs).
%
%
Kostrykin el al. \cite{Kostrykin} constructs Brownian motion on metric graphs identifying the generator of the process.
%
Fitzsimmons and Kuter study the harmonic functions of the Walsh Brownian motion in \cite{FitzsimmonsKuter} and on metric graphs in  \cite{fk}.

For a more recent work see the paper by Karatzas and Yan \cite{KaratzasYan} where they generalize the Walsh Brownian motion to a semimartingale framework,
provide an It\^o type formula and generalize other classical tools of stochastic calculus (as the removal of drift by a change of the measure). They furthermore characterize the solution of an optimal stopping problem in this framework, and discuss also 
stochastic control problems with discretionary stopping. 
An optimal stopping problem on a Brownian spider, related to maximal inequalities, has been considered by Ernst \cite{Ernst}.
We refer also to Bayraktar and Zhang \cite{BayZhang} for an embedding problem for Walsh Brownian motions. For local times and occupation times for random walks and Brownian spiders, see  Cs\'aki et al. \cite{CCFP1} and \cite{CCFP2}.  
In Dassios and Zhang \cite{DassiosZhang} further results on hitting times for Brownian spiders can be found, 
with a potential application in mathematical finance, more precisely related to a liquidity risk problem in a system of banks.

In the present paper, following \cite{BarlowPitmanYor1}, we consider the case where instead of a Brownian motion on the legs a regular recurrent diffusion is used.
We call such a process a \emph{diffusion spider}. 

The contribution of the paper is twofold: 
Firstly, we derive (hereby our basic reference is \cite{Rogers}) a formula for the Green function (resolvent density)  of this process. This is used, in particular,  to characterize the excessive functions of the diffusion spider. For results on  excessive functions for Walsh Brownian motion, see also \cite{BayZhang} and \cite{KaratzasYan}.  Secondly, we solve some optimal stopping problems. 
For the solution of these optimal stopping problems we use two closely related but different methods,  
inscribed in the so called \emph{representation approach}:
(i) when the stopping region is expected to be connected, the candidate for the representing measure of the value function is used in order to find the solution,  
(ii) when the stopping region is expected to be  disconnected, a verification theorem is adapted from \cite{CCMS}, based on the notion of the extension of the negative set of the problem (see also \cite{CrocceMordecki}). 
Both methods above require the computation of the Martin kernel, or equivalently the Green kernel. 
It is very well established that the resolvent kernel (i.e. the integration against the Green function)
is the inverse operator to the infinitesimal generator of the process. 
This second operator is described locally inside the edges as a generalized differential operator obtained from the infinitesimal generator of the underlying diffusion. To complete the description of the infinitesimal generator (and the process) we should determine its domain. In our approach to optimal stopping we use the relationship between the resolvent operator and the infinitesimal generator as a tool but circumvent the consideration of the domains of the operators. 

The rest of the paper is organized as follows. 
In Section \ref{S} we introduce the diffusion spider and discuss its existence as a strong Markov process.
After this, in Section \ref{L1} some preliminaries on linear diffusions are presented.
In the following Section \ref{GR} the first main results of the paper are given, that is, an expression for the density of the Green kernel of the spider.
We proceed then in Section \ref{HT} by computing the Laplace transform of a hitting time associated with a diffusion spider.
Section \ref{EX} contains central main results of the paper concerning  the representing measure of excessive functions. Other characterizations of excessive functions are also discussed.
In Section \ref{O} we proceed with general results on optimal stopping, applied to diffusion spiders, adapted from \cite{CCMS} and \cite{CrocceMordecki}.
In Section \ref{OSP} we solve three examples of optimal stopping for a Brownian spider. 
The first one has a connected stopping region and the solution is obtained by analyzing the representing measure formulas for the candidate value function. 
The other two examples have disconnected stopping regions and are solved once the existence and the uniqueness of the solution of the associated equation system is established using the results of Section  \ref{O}. 
We conclude with two appendices; in particular, the second  is about a smooth fit property for diffusion spiders.

\bigskip
\section{Definition and existence of a diffusion spider}
\label{S}
\bigskip
Let $X=\{X_t\colon t\geq 0\}$ be a recurrent  linear (Feller) diffusion  taking values in $[0,\infty)$. It is assumed that $0$ is a regular boundary point with reflection, and  $+\infty$ is  natural (for the boundary classification of a linear diffusion, see, e.g.,  \cite{BorodinSalminen}).  The notations $\P_x$ and $\EE_x$ are for the probability measure and the expectation operator, respectively, associated with $X$ when initiated at $x.$

Consider a "star shaped" graph $\Gamma\subset \R^2$ with only one internal vertex taken to be the origin of $\R^2$, and $n$ edges $\ell_1,\dots,\ell_n$, also called legs, meeting each other the origin $\0$. Hence, the legs are of infinite length. For $x\geq 0$ and $i\in\{1,2,\dots,n\}$ we let  $\x=(x,i)$ denote the point on $\Gamma$ located on the leg $\ell_i$ at the distance $x$ from the origin. For any $i$ the point $(0,i)$ is identified as the origin. The topology on $\Gamma$ is the relative Euclidean topology, and, hence,  $\Gamma$ is a locally compact Hausdorff space  with a countable base.

On the graph $\Gamma$ we consider a stochastic process  for which we use the notation 
$$
\X:=\{\X_t:=(X_t,\rho_t)\colon t\geq 0\},
$$ 
where $\rho_t\in\{1,2,\dots,n\}$ indicates
the leg on which $\X_t$ is located at time $t$ and $X_t$ is  the distance of $\X_t$ to the origin at time $t$ measured along the leg $\rho_t$. On each edge $i=1,\dots,n$ before hitting zero, the process $\X$ behaves like the diffusion $X$.
As a part of  the definition of the process 
there is a vector $\p=(p_1,\dots,p_n)$  such that $p_i>0\ (i=1,\dots,n)$ and $\sum_{i=1}^np_i=1$.
When  $\X$ hits $\0$, it  continues, roughly speaking, with probability $p_i$ 
on the leg $\ell_i$. 
In Figure \ref{figure:3} we have illustrated such a graph with 3 legs.  
\begin{figure}[H]
\centering
\begin{tikzpicture}[scale=0.7, transform shape]
\tikzstyle{every node} = [circle, fill=gray!0,draw]
\node (0) at (0,0) {$\0$};
\tikzstyle{every node} = [circle, fill=gray!0]
\node (1) at (0,3) {$\ell_1$};
\node (11) at (0,1.5) {$p_1$};
\node (2) at (-2.6,-1.5) {$\ell_2$};
\node (22) at (-1.3,-0.75) {$p_2$};
\node (3) at (2.6,-1.5) {$\ell_3$};
\node (33) at (1.3,-0.75) {$p_3$};
\foreach \from/\to in {11/1,22/2,33/3}
\draw [->,line width=0.4mm, fill=black] (\from) -- (\to);
\foreach \from/\to in {0/11,0/22,0/33}
\draw [-,line width=0.4mm, fill=black] (\from) -- (\to);
\end{tikzpicture}
\caption{A star graph with $3$-legs}\label{figure:3}
\end{figure}
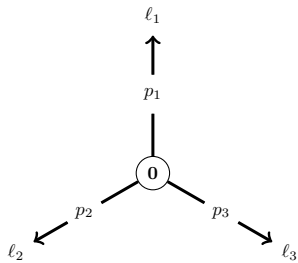

This process when  $X$ is a reflected Brownian motion was introduced by Walsh in \cite{Walsh}, and is, hence, in this case, called the Walsh Brownian motion with discrete (finitely many atoms) distribution for the angles. 
A formal construction of the Walsh Brownian motion  based on the semigroup theory was given by Barlow et al. in \cite{BarlowPitmanYor1}. We refer to \cite{BarlowPitmanYor1} also for earlier references. According to the remark on p. 281 in \cite{BarlowPitmanYor1}, a similar construction can be carried out in our case where the underlying process is $X$, i.e., a general recurrent Feller diffusion on $[0,\infty)$ reflected at zero. We call this process  a (homogeneous)  diffusion spider. It is a strong Markov process with a Feller semigroup and continuous sample paths (in the relative Euclidean topology on $\Gamma$). 

To indicate explicitly the number of legs and the probabilities for the excursions we write  $\X=(X,n,\p)$.  The notations $\P_{(x,i)}$ and $\EE_{(x,i)}$ are for the probability measure and the expectation operator, respectively, associated with $\X$ when initiated at $(x,i).$
The term homogeneous refers to the property  that on each leg the movement  is governed by the same diffusion $X$. The non-homogenous case is the one when the diffusion is not the same on each leg.  This situation and also more general graphs has been considered in   \cite{FreidlinWentzell}.
\bigskip
\section{Preliminaries for linear diffusions}
\label{L1}
\bigskip
Consider the  linear diffusion $X$ introduced above. Recall  that the generator  of $X$  is given in the form 
\begin{equation}\label{gen}
\cG u:=\frac d{dm}\frac d{dS}u,
\end{equation}
where $m$ and $S$ denote the  speed measure and the scale function of $X$, respectively, and $u$ is a function in the domain of the generator. Recall that 0 is a regular boundary point with reflection. For simplicity, it is  assumed that $m(\{0\})=0$. We also put $S(0)=0.$ 

In addition to the diffusion $X$, we need also the diffusion $X^{\partial}$ with the same speed and scale as $X$ but for which 0 is a killing boundary. For $r>0$ let  $\varphi_r\ (\varphi_r^{\partial})$ and $\psi_r\ (\psi_r^{\partial})$ be the decreasing and the increasing, respectively, fundamental solutions of the generalized differential equation
$$
\frac d{dm}\frac d{dS}f=rf
$$
associated with  $X\ (X^{\partial})$, see, e.g., \cite{BorodinSalminen}. These functions are non-negative, continuous and unique up to  multiplicative constants. For simplicity, it is moreover assumed that $\varphi_r$ and $\psi_r^\partial$ are differentiable with respect to $S$.  To shorten the notation,  the subindex $r$ is omitted and 
it is written, e.g., $\varphi$ instead of $\varphi_r$. 

Recall also that 
\begin{equation}\label{recall}
\psi^{\partial}(0)=0, \quad{d\psi\over d S}(0+)=0,\quad \varphi^\partial\equiv\varphi,
\end{equation}
and for $x\geq 0$ 
\begin{equation}
\label{H0}
\EE_x(e^{-rH_0})=\frac{\varphi(x)}{\varphi(0)},
\end{equation}
where $H_y:=\inf\{t\,:\, X_t=y\}$ with $y=0$. 
Moreover, we normalize 
\begin{equation}
\label{norm}
\varphi(0)=\varphi^\partial(0)=1,\quad \text{and}\quad {d\psi^\partial\over d S}(0+)=1,
\end{equation}
and, consequently, the Wronskian $w_r$ of $X^\partial$ is given by
\begin{equation}
\label{w0}
w_r:={d\psi^\partial\over d S}(x)\varphi(x)-{ d\varphi\over d S}(x)\psi^\partial(x)
=
{ d\psi^\partial\over d S}(0+)=1,
\end{equation}
where it is used that the Wronskian does not depend on $x$ and that 
\begin{equation}
\label{reg00}
{ d\varphi\over d S}(0+)>- \infty,
\end{equation}
since 0 is assumed to be a regular boundary point. Finally, because $+\infty$ is assumed to be natural,  it holds
\begin{equation}
\label{p0}
\lim_{x\to +\infty}\frac{d\varphi}{dS}(x)=0.
\end{equation}

\bigskip
\relax
\section{Green kernel}
\label{GR}
\bigskip
Consider now a diffusion spider  $\X=(X,n,\p)$ with $n$ legs and the corresponding probabilities $p_1,p_2,...,p_n$. Define  
\begin{align*}
&\mathbf{m}(dx,i):=p_im(dx),\quad i=1,\dots,n,\qquad \mathbf{m}(\{\0\}):=0, \\
&\mathbf{S}(dx,i):=\frac{1}{p_i}S(x).
\end{align*}
We call $\mathbf{m}$ and $\mathbf{S}$  the speed measure and the scale function, respectively, of $\X$. Clearly, on every leg $\ell_i,i=1,\dots,n$, of the spider it holds 
$$
\frac d{d\mathbf{m}}\frac d{d\mathbf{S}}=\frac d{dm}\frac d{dS}.
$$

The main contribution in this section is  an explicit expression of the Green kernel (also called the resolvent kernel) for $\X$ defined via 
\begin{equation}
\label{green11}
{\bf G}_rf(x,i):=\int_0^{\infty}e^{-rt}\EE_{(x,i)}\left(f(\X_t)\right)\,dt,
\end{equation}
where $r>0$ and $f\colon\Gamma\to\R$ is a bounded measurable function. 

\relax

\begin{theorem}
\label{Thrm:GreenKernel1}
The Green  kernel of the diffusion spider $\X$ has a density ${\bf g}_r$ with respect to  the speed measure {\bf m} 
which is 
given for $ x\geq 0$ and $y\geq 0$ by
\begin{equation}
\label{gk1}
{\bf g}_r((x,i),(y,j))=
\begin{cases}
\vspace{2mm}
\varphi(y)\tilde{\psi}(x,i)
,&\quad  x\leq y,\ i=j\\
\vspace{2mm}
\varphi(x)\tilde{\psi}(y,i)
,&\quad  y\leq x,\ i=j\\
{\displaystyle c_r^{-1} 
\,\varphi(y)
\varphi(x)},&\quad i\not=j,
\end{cases}
\end{equation}
where 
\begin{equation}
\label{tildepsi}
\tilde{\psi}(x,i):=\frac1{p_i}\psi^{\partial}(x)+\frac1{c_r}\varphi(x),
\end{equation}
and 
\begin{equation}
\label{cr1}
c_r:=-{d\over dS}\varphi(0+)>0.
\end{equation}
\end{theorem}
\relax
\begin{proof}
We apply the strong Markov property of $\X$ at  
the first hitting time 
$$
\overline {H}_\0:=\inf\{t\geq 0\colon \X(t)=\0\}
$$ 
of the origin to obtain the following equation for the Green kernel (see Lemma 1 in \cite{Rogers})  
\begin{equation}\label{res}
\G_rf(x,i)=\G^\partial_rf(x,i)+\EE_{(x,i)}(e^{-r\HH_\0})\G_rf(\0),\quad x\geq 0, i\in\{1,2,...,n\},
\end{equation}
where $\G^\partial_r$ is the Green kernel  of $\X$  killed at $\0$ and $f$ is a real-valued, bounded and measurable function defined on $\Gamma$.  Since $\X$ behaves on the legs before hitting $\0$ as the diffusion $X$ before hitting 0 the Green kernel $\G^\partial_r$ coincides on a given leg with the Green kernel of $X$ killed at 0. Hence,  for all $i$ 
$$
\G^\partial_rf(x,i)=\int_0^\infty g_r^\partial(x,y)\, f(y,i)\, m(dy),
$$
where  
\begin{equation}
\label{gr0}
g_r^\partial(x,y)=
\begin{cases}
\vspace{2mm}
\psi^\partial(y)\varphi(x),&0\leq y\leq x,\\
\psi^\partial(x)\varphi(y),&0\leq x\leq y,
\end{cases}
\end{equation}
see \cite{BorodinSalminen} p. 19 and recall (\ref{w0}).
Consequently, $\G^\partial_r$ has a density $\g^\partial_r$ (with respect to $\bf m$) given by 
\begin{equation}
\label{gr1}
\g^\partial_r((x,i),(y,j))=
\begin{cases}
\vspace{2mm}
0,&\text{  $i\neq j$}\\
 g^\partial_r(x,y)/p_i,&\text{ $i=j$ } 
\end{cases}
\end{equation}
Consider now the second term on the right hand side of (\ref{res}). Clearly, for all $i$
\begin{equation}
\label{gr2}
\EE_{(x,i)}(e^{-r\HH_0})=\EE_{x}(e^{-r H_0})=\varphi(x),
\end{equation}
where  (\ref{H0}) is used. To find $\G_rf(\0)$ recall formula (2) in \cite{Rogers} saying that 
\begin{equation}\label{eq:r2}
\G_rf(\0)={{\boldsymbol{\nu}}_rf\over r{\boldsymbol{\nu}}_r1}.
\end{equation}
Here ${\boldsymbol{ \nu}}_r$ 
refers for any $r>0$ to a  non-negative $\sigma$-finite measure obtained as the 
Laplace transform of the excursion entrance law associated with $\X$ for excursions starting from $\0$. The excursions from $\0$ of the diffusion spider take place on the legs, 
and we let $(\xi,i)$ denote such an excursion on the $i$th leg with $\xi$ as an element in the space
\begin{align*}
&U:=\{\xi\colon[0,\infty)\to[0,\infty)\,:\, \xi \text{ continuous, }\,  \xi(0)=0, \\
&\hskip2cm \text{ and } \exists\, \zeta(\xi)>0 \text{ such that } \xi(0)>0\text{ for } 0<t<\zeta(\xi) \text{ and } \xi(t)=0\text{ for }t\geq\zeta(\xi)
\}.
\end{align*} 
Let $\nu$ denote the It\^o excursion measure of the linear diffusion $X$ for excursions starting from 0, see Revuz and Yor \cite{RevuzYor} Chapter XII,
that is 
$$
\nu^t(A):=\nu(\{\xi\in U\colon \xi(t)\in A,0<t<\zeta(\xi)\}),\  A\in \mathcal{B}([0,\infty)),
$$
and it holds (cf. Salminen, Vallois and Yor \cite{salminenvalloisyor07})
$$
\nu^t(A)=\int_Ah(t,y)\,{m}(dy),
$$
where 
$$
h(t,y):=\P_{y}(H_0\in dt)/dt.
$$
For the Laplace transform of $ \nu^t(A)$ we have
$$
\nu_r(A):=\int_0^\infty \e^{-rt}\nu^t(A)\,dt=\int_A{m}(dy)\int_0^\infty dt\, \e^{-rt}h(t,y)
=\int_A \varphi(y)\,{m}(dy).
$$
Let ${\displaystyle{\overline {A}:=\cup_{i=1}^n(i,A_i)}}$ with $A_i\in \mathcal{B}([0,\infty)),\, i=1,2,...,n,$ be a subset of $\Gamma$.  Then the excursion entrance law ${\boldsymbol{\nu}}^t$ of $\X$ has the representation
$$
{\boldsymbol{\nu}}^t(\overline {A})=\sum_{i=1}^n p_i\,\nu^t(A_i).
$$
Consequently,
$$
{\boldsymbol{\nu}}_r(\overline {A}):=\sum_{i=1}^n p_i\,\nu_r(A_i)=\sum_{i=1}^np_i\int_{A_i}\varphi(y)\,m(dy),
$$
and putting here $A_i=(0,\infty),\, i=1,2,...,n,$ yields
$$
\aligned
r{\boldsymbol{\nu}}_r 1&=r\sum_{i=1}^np_i\int_0^\infty\varphi(y)m(dy)=r\int_0^\infty\varphi(y)m(dy)\\
&=\int_0^\infty{d\over dm}{d\over dS}\varphi(y)m(dy)=-{d\over dS}\varphi(0+)=c_r,
\endaligned
$$
where  (\ref{p0}) and \eqref{norm} are used. 
Hence, for the real-valued, bounded and measurable function $f$ defined on $\Gamma$ we may now write, cf. (\ref{eq:r2}),
$$
\G_rf(\0)=\frac 1{c_r} \sum_{i=1}^np_i\int_0^\infty\varphi(y)f(y,i)m(dy),
$$
and the density with respect to $\mathbf{m}(dx,i)$ is then
\begin{equation}
\label{gr3}
\g_r(\0,(y,i))
=\frac1{c_r}\varphi(y),
\end{equation}
which does not depend on $i$.  
Combining (\ref{gr1}), (\ref{gr2}), and (\ref{gr3}) results to the formula 
\begin{equation*}
\g_r((x,i),(y,j))=\frac1{p_j}\psi^\partial(x)\varphi(y)\indicator_{\{i=j\}}+\frac1{c_r}\varphi(x)\varphi(y),
\end{equation*}
concluding the proof.
\end{proof}

\relax

It is intresting and useful to note that to the function $\tilde\psi$ appearing in the formula for the Green function has the following property. 

\begin{proposition}
\label{prop00}
 For all $i=1,\dots,n$, the function
$$
\tilde{\psi}(x):=\tilde{\psi}(x,i)=\frac1{p_i}\psi^\partial(x)+\frac1{c_r}\varphi(x),\quad x\geq 0,\ 0<p_i<1,
$$
is positive and increasing.
\end{proposition}
\relax
\begin{proof}
The positivity of $\tilde \psi$ is immediate from the definitions of $\psi^\partial$ and $\varphi$.  For the proof that $\tilde \psi$ is increasing  we assume, for simplicity,  that $m$ and $S$ have positive and continuous derivatives. The task is to show that  
\begin{equation}\label{c1}
{d\tilde{\psi}\over d S}(x)={1\over S'(x)}{d\tilde{\psi}\over d x}(x)>0.
\end{equation}
For this, consider
$$
{d\tilde{\psi}\over d S}(x)=\frac1{p_i}{d\psi^{\partial}\over dS}(x)+\frac1{c_r}{d\varphi\over dS}(x)
=\frac1{p_i}{d\psi^{\partial}\over dS}(x)-\frac1{{d\varphi\over dS}(0+)}{d\varphi\over dS}(x)
$$
Taking limit as $x\downarrow 0$, we obtain
$$
{d\tilde{\psi}\over d S}(0+)
=\frac1{p_i}{d\psi^{\partial}\over dS}(0+)-\frac1{{d\varphi\over dS}(0+)}{d\varphi\over dS}(0+)
=\frac1{p_i}-1>0.
$$
Differentiate with respect to $m$ to obtain
$$
{d\over dm}{d\tilde{\psi}\over dS}(x)=\frac1{p_i}{d\over dm}{d\psi^{\partial}\over dS}(x)+\frac1{c_r}{d\over dm}{d\varphi\over dS}(x)=\frac{r}{p_i}\psi^{\partial}(x)+\frac{r}{c_r}\varphi(x)>0.
$$
This means that $x\mapsto{d\tilde{\psi}\over dS}(x)$ is increasing. Therefore, 
$$
0<{d\tilde{\psi}\over d S}(0+)<{d\tilde{\psi}\over d S}(x),
$$ 
and the inequality in (\ref{c1})  follows. Hence,  ${d\tilde{\psi}\over d x}(x)>0$ since  $S'(x)>0$ by the assumption. 
\end{proof}

\begin{remark}
\label{exxx0}
{\sl
    A diffusion spider with $2$ legs, and $\p=(p,1-p)$  is probabilistically  equivalent with a skew diffusion on $\R$ when we identify, for instance,  the leg 2  with the negative half line. Let  $X^{(sk)}$ denote this diffusion with the fundamental decreasing and increasing solutions $\varphi^{(sk)}$  and $ \psi^{(sk)}$, respectively. Using the Green function given in  Theorem \ref{Thrm:GreenKernel1} we can derive the following expressions for  $\varphi^{(sk)}$  and $ \psi^{(sk)}$
    $$
    \varphi^{(sk)}(x)=
    \begin{cases}
        c_r\tilde\psi(-x,2),& x\leq 0,\\
        \varphi(x),&x\geq 0,
    \end{cases}
    \qquad
    \psi^{(sk)}(x)=
    \begin{cases}
        \varphi(-x),&x\leq 0,\\
        c_r\tilde\psi(x,1),& x\geq 0.
    \end{cases}
$$
Notice that from these relationships we may also  deduce the claim in Proposition \ref{prop00}. 
}
    
\end{remark}

\begin{example}
\label{exxx1}
{\sl
Consider a Brownian spider $\X$ with $n$ legs and the corresponding  probabilities $p_1,p_2,\dots,p_n$. The  underlying diffusion $X$ is a reflecting Brownian motion and the scale function and speed measure of $X$ are
$S(x)=x$ and $m(dx)=2dx$. Morover, 
$$
\varphi(x)=\varphi^\partial(x)=e^{-\theta x},\qquad \psi(x)=\frac1{\theta }\cosh(\theta x), \qquad
\psi^\partial(x)=\frac1{\theta }\sinh(\theta x),
$$
where $\theta=\sqrt{2r}$. These functions meet the normalizations made in (\ref{norm}). We have
$$
c_r=-{d\over dS}\varphi(0+)=\theta,
$$
and, hence, 
\begin{equation}\label{tildepsi1}
\tilde{\psi}(x,i)=\frac1{p_i}\frac1{\theta}\sinh(\theta x)+\frac1{\theta}e^{-\theta x}.
\end{equation}
The Green function, see Theorem \ref{Thrm:GreenKernel1},  is in this case given by 
\begin{equation}\label{kernel}
\g_r((x,i),(y,j))=
\begin{cases}
\vspace{2mm}
	{1\over\theta}{\e^{-\theta x}}\left(\frac1{p_i}\sinh(\theta y)+\e^{-\theta y}\right),&  \ 0\leq y\leq x,\ i=j,\\
\vspace{2mm}
{1\over\theta}\e^{-\theta y}\left(\frac1{p_i}\sinh(\theta x)+\e^{-\theta x}\right),& \ 0\leq x\leq y,\ i=j,\\
	{1\over\theta}\e^{-\theta x}\e^{-\theta y},&  i\neq j.\\
\end{cases}
\end{equation}
Notice that $\g_r(\0,\0)=1/\theta$ and that the Green function is  continuous everywhere (and in particular at $\0$).
It is possible to obtain the density function of the  Brownian spider with respect to  the speed measure ${\bf m}$ by inverting the Laplace transforms above.
In fact, as the Green kernel
is a linear combination of the Green kernels of the killed BM and the standard BM, the corresponding densities
can be obtained through the same linear combination. We obtain, when $i\neq j$, that
$$
p(t;(x,i),(y,j))={2p_j\over\sqrt{2\pi t}}\e^{-{(x+y)^2\over 2t}},\qquad i\neq j,
$$
and for $i=j$ 
$$
p(t;(x,i),(y,i))={1\over\sqrt{2\pi t}}\left(\e^{-{(x-y)^2\over 2t}}
-\e^{-{(x+y)^2\over 2t}}
\right)+
{2p_i\over\sqrt{2\pi t}}\e^{-{(x+y)^2\over 2t}}.
$$}
\end{example}

\bigskip
\section{Hitting times}\label{HT}
\bigskip
Let $\X=(X,n,\p)$ be  a homogeneous diffusion spider and  $\y=(y,1)$ a fixed point on $\Gamma$. Define the hitting time
$$
H_\y=\inf\{t\geq 0\colon \X_t=\y\}.
$$
The objective here is find $\EE_\x\lp e^{-rH_\y} \rp$, the Laplace transform of $H_\y$.
Let $\x=(x,i)$ be the starting state of $\X$ and assume first  that $i\neq 1$. In this case we observe that by the strong Markov property it holds 
$$
\EE_\x\lp
e^{-rH_\y}
\rp=
\EE_\x\lp
e^{-r(H_\y\circ\theta_{H_\0}+H_\0)}
\rp
=
\EE_\x\lp
e^{-rH_\0}\rp
\EE_\0\lp
e^{-rH_\y}
\rp={\varphi(x)\over\varphi(0)}
\EE_\0\lp
e^{-rH_\y}\rp, 
$$
where $\theta_t$ denotes the usual shift operator. 
To calculate  $\EE_\0\lp e^{-rH_\y}\rp$ we consider an auxiliary skew diffusion $X^{(sk)}$ defined on $\R$ (see Remark \ref{exxx0}), such that
$$
\P_{0}(X_t^{(sk)}\leq 0)=1-p_1,\qquad
\P_{0}(X_t^{(sk)}\geq 0)=p_1.
$$
Then, since the spider diffusion behaves as $X$ on each leg before hitting $\0$ we have 
$$
\EE_\0\lp e^{-rH_\y}\rp=\EE_0^{(sk)}\lp e^{-rH_y}\rp={\psi^{(sk)}(0)\over\psi^{(sk)}(y)}
$$
where $\psi^{(sk)}$ is the increasing fundamental solution associated with the skew diffusion $X^{(sk)}$. We conclude that
\begin{equation}
\label{lemma1}
\EE_\x\lp
e^{-rH_\y}
\rp=
\begin{cases}
\vspace{2mm}
\displaystyle{{\varphi(x)\over\varphi(0)}{\psi^{(sk)}(0)\over\psi^{(sk)}(y)}},& \text{$0\leq x$ on leg $i\neq 1$},\\
\vspace{2mm}
\displaystyle{\psi^{(sk)}(x) \over\psi^{(sk)}(y)},& \text{$0\leq x\leq y$ on leg $1$},\\
\displaystyle{\varphi(x)\over\varphi(y)},& \text{$0\leq y\leq x$ on leg $1$}.\\
\end{cases}
\end{equation}
We remark that $\psi^{(sk)}$ can be expressed in terms of the functions $\psi^\partial$ and $\varphi$. Indeed, $\psi^{(sk)}$ defined on $\R$ should be positive, continuous and increasing. Morover, it should satisfy for $x\not= 0$ the (generalized) ODE 
$$
\frac{d}{dm}\frac{d}{dS}u=ru
$$
and at 0 the condition
$$
(1-p_1)\frac{du}{dS}(0-)=p_1\frac{du}{dS}(0+).
$$
Letting $\hat\varphi(x):=\varphi(-x), x\leq 0$, it is easily seen that 
\begin{equation}
\label{skewup}
\psi^{(sk)}(x)=
\begin{cases}
\vspace{2mm}
\tilde{\psi}(x,1)=\frac1{p_1}\psi^{\partial}(x)+\frac1{c_r}\varphi(x),& x\geq 0,\\
\vspace{2mm}
\frac1{c_r}\hat\varphi(x)=\frac1{c_r}\varphi(-x),& x\leq 0,\\
\end{cases}
\end{equation}
is the desired function (cf. Remark \ref{exxx0}).
\begin{example}
\label{hitskew1} {\sl For  a Brownian spider we have
$$
\varphi(x)=e^{-\theta x},\quad x>0,
$$
$$
\psi^{(sk)}(x)={1-2p_1\over p_1}\sinh(\theta x)+e^{\theta x},\quad x>0,
$$
where $\varphi$  $(\psi^{(sk)})$ are,  for $x>0$, the decreasing (the increasing)  fundamental solution for the skew Brownian motion with parameter $\beta=p_1$, see \cite{BorodinSalminen} A1.12 p. 130.
Observe that
$$
\psi^{(sk)}(x)={1-2p_1\over p_1}\sinh(\theta x)+e^{\theta x}=\frac1{p_1}\sinh(\theta x)-2\sinh(\theta x)+e^{\theta x}=\frac1{p_1}\sinh(\theta x)+e^{-\theta x}
$$
and , hence, $\psi^{(sk)}$ equals with the (unnormalized) function $\tilde\psi$ appearing in \eqref{tildepsi1}, as stated in (\ref{skewup}).} 
\end{example}

\bigskip
\relax
\section{Excessive functions}
\label{EX}
\bigskip
Recall that for a given  $r\geq 0$  the measurable function $f\colon\Gamma\to[0,\infty)\cup\{+\infty\}$ is called $r$-excessive for $\X$  if for all $(x,i)\in\Gamma$
\begin{equation}
\label{ex1}
\lim_{t\to 0}\e^{-rt}\EE_{(x,i)}\lp f(\X_t)\rp \uparrow f(x,i).
\end{equation}
Since $\X$ is regular in the sense that every point of $\Gamma$ is hit with a positive probability it follows from (\ref{ex1}) that if 
$(x,i)\in\Gamma$ is such that for an excessive function $f$ it holds $f(x,i)=0$ $(f(x,i)=+\infty)$ then $f\equiv 0$  $(f\equiv +\infty)$. These functions are called trivial $r$-excessive functions. 

We study the non-trivial $r$-excessive functions of $\X$ via the Martin boundary theory. Using the Martin metric induced by the Green function $\g_r$ the state space $\Gamma$ of $\X$ is compactified to  $\bar\Gamma:=\Gamma\cup\{(\infty,i):\, i=1,2,...,n \}$ (the Martin compactification).  For every $(y,k)\in\bar\Gamma$ 
define an $r$-excessive function by 
$$
(x,i)\mapsto{\bf u}_r((x,i);(y,k)):= 
\begin{cases}
\vspace{2mm}
\displaystyle{\frac{\g_r((x,i),(y,k))}{\g_r(\0,(y,k))}},& 0\leq y<\infty,\\
\displaystyle{\lim_{z\to\infty} \frac{\g_r((x,i),(z,k))}{\g_r(\0,(z,k))}},& y=\infty.
\end{cases}
$$
These functions are called the minimal $r$-excessive functions of $\X$ attaining the value 1 at $\0$.  
Let $\sigma$ be a probability measure on $\Gamma$ and define for $(x,i)\in\Gamma$
\begin{equation}
\label{ma}
   h(x,i;\sigma) := \sum_{k=1}^n \int_{(0,\infty]}{\bf u}_r((x,i);(y,k))\,\sigma(dy,k)  +
{\bf u}_r((x,i);\0)   \,\sigma(\{\0\}).
\end{equation}
Notice that  $h(\0;\sigma)=1$ since $\sigma $ is assumed to be a probability measure.   It can be proved directly that $h$ is $r$-excessive for $\X$ but this also follows from the Martin boundary theory. In fact, there is a one-to-one correspondence between the probability measures on $\bar\Gamma$ and the non-trivial $r$-excessive functions of $\X$. For this, see \cite{KunitaWatanabe} Theorem 4 p. 513 which is proved under Hypothesis (B) p. 498 valid in our case since the Green function is symmetric (and, hence,  $\X$ is self-dual). 

Consider next 
\begin{equation}\label{mahar1}
\vspace{2mm}
{\bf u}_r((x,i);(\infty,k))=\lim_{z\to\infty} \frac{\g_r((x,i),(z,k))}{\g_r(\0,(z,k))}=
\begin{cases}
\vspace{2mm}
c_r \tilde\psi(x,i),& i=k,\\
\varphi(x),& i\not= k,
\end{cases}
\end{equation}
where $\tilde \psi$ is given in (\ref{tildepsi}). The function $(x,i)\mapsto {\bf u}_r((x,i);(\infty,k))$ is $r$-invariant (and also $r$-harmonic)  for $\X$ since its representing measure   is supported by the infinite part of the Martin compactification. 
Moreover, all (positive) $r$-harmonic functions can be written as linear combinations with non-negative coefficients of 
${\bf u}_r((x,i);(\infty,k)), k=1,2,\dots,n$.  Therefore, given a positive $r$-harmonic function $H$ there exists $a_i\geq~0, i=1,2,\dots,n,$ such that  
\begin{equation}
\label{mahar2}
H(x,i)=\sum_{k=1}^n a_k\,{\bf u}_r((x,i);(\infty,k))=a_ic_r\tilde\psi(x,i)+\sum_{k=1, k\not=i}^n a_k\varphi(x)=
\frac{a_ic_r}{p_i}\psi^\partial(x)+\varphi(x)\sum_{k=1}^n a_k.
\end{equation}
Clearly,  $(\e^{-rt}H(\X_t))_{t\geq0}$ is a positive martingale and $H(\0)=a_1+\cdots+ a_n$. Notice also that since both $\psi^\partial$ and $\varphi$ are $r$-harmonic for $X^\partial$, it follows that  
 $H$ when restricted on the leg $i$ is $r$-harmonic for $X^\partial$ on that leg. However, there are harmonic functions of $X^\partial$ which are not harmonic for $\X$, e.g., $\psi^\partial$. In Theorem \ref{charc2} this issue is taken up more in detail.   

\begin{example}
\label{exam1}
{\sl In the particular case of the Brownian spider, the general form of a harmonic function is 
\begin{equation}\label{mahar23}
H(x,i)=
\frac{a_ic_r}{p_i}\psi^\partial(x)+\varphi(x)\sum_{k=1}^n a_k= \frac {a_i}{p_i}\sinh(\theta x)+e^{-\theta x}\sum_{k=1}^n a_k,
\end{equation}
since
$$\varphi(x)=e^{-\theta x},\qquad \psi^\partial(x)=\frac1{\theta }\sinh(\theta x),\qquad
c_r=-{d\over dx}\varphi(0+)=\theta=\sqrt{2r} .
$$
}
\end{example}

\relax

In the next theorem formula (\ref{glue1}) states a fundamental property $r$-excessive functions found earlier in the litterature, see, for instance, Freidlin and Sheu \cite{fs} p. 183 (where (\ref{glue1}) is called the glueing condition), 
 Thm. 4.1 in Fitzsimmons and Kuter \cite{FitzsimmonsKuter}, and  
formula (5.7) in Karatzas and Yan  \cite{KaratzasYan}. In \cite{FitzsimmonsKuter} and  \cite{KaratzasYan}  a more general process with uncountable many rays is considered.  
The proof below is based on the Martin representation (\ref{ma}) and extends the proof from the one-dimensional case as given in Salminen \cite{Salminen} (3.7) Corollary.     

\begin{theorem}
\label{glue}
Let $f$ be a non-trivial $r$-excessive function of $\X$. Assume that $f(\0)=1$ and let $\sigma$ be the probability measure such that the right hand side of (\ref{ma}) represents $f$. Then 
$$
f^+(\0,i):=\lim_{\delta \downarrow 0}  \frac{f(\delta,i)-f(\0)}{S(\delta)-S(0)}
$$
exists for every $i=1,2,...,n$, and 
\begin{equation}\label{glue1}
\mathbf{D} {f}(\0):=\sum_{i=1}^np_i f^+(\0,i) \leq 0.
\end{equation}
Moreover, for $x > 0$ and  $i=1,2,...,n$
$$
f^-(x,i):=\lim_{\delta \downarrow 0}  \frac{f(x,i)-f(x-\delta,i)}{S(x)-S(x-\delta)}\ \ {\text and}\ \  
f^+(x,i):=\lim_{\delta \downarrow 0}  \frac{f(x+\delta,i)-f(x,i)}{S(x+\delta)-S(x)}.
$$
exist, and satisfy
\begin{equation}\label{glue2}
f^+(x,i) - f^-(x,i)=-\frac{c_r}{p_i}\frac{\sigma(\{x\},i)}{\varphi(x)} \leq 0.
\end{equation}
\end{theorem}
\relax
\begin{proof}
We assume that (\ref{ma}) holds and write
\begin{align*}
f(x,i) &
	=\int_0^{+\infty}\frac{\g_r((x,i),(y,i))}{\g_r(\0,(y,i))}\,\sigma(dy,i) \\
		&\hskip2cm+\sum_{k\not= i} \int_{0}^{+\infty}\frac{\g_r((x,i),(y,k))}{\g_r(\0,(y,k))}\,\sigma(dy,k) +  \frac{\g_r((x,i),\0)}{\g_r(\0,\0)}\,\sigma(\0)\\
&
=:\Delta_ 1+\Delta_2+\Delta_3.
\end{align*}
Consider first $\Delta_1$. Plug in the explicit expression for the resolvent density (Green function) from (\ref{gk1}) and recall $\varphi(0)=1$ to obtain
\begin{align*}
&\Delta_1=\int_{(0,x]}\frac{\lp\frac 1{p_i}\psi^\partial(y)+\frac 1{c_r}\varphi(y)\rp\varphi(x)}{\frac 1{c_r}\varphi(y)} \sigma(dy,i) +
\int_{(x,\infty)}\frac{\lp\frac 1{p_i}\psi^\partial(x)+\frac1{c_r}\varphi(x)\rp\varphi(y)}{\frac 1{c_r}\varphi(y)} \sigma(dy,i)\\
&\hskip.5cm =\frac{c_r}{p_i}\lp\varphi(x)\int_{(0,x]}\frac{\psi^\partial(y)}{\varphi(y)}\,  \sigma(dy,i)+\psi^\partial(x)\sigma((x,\infty),i)\rp+\varphi(x)\sigma((0,x],i) +\varphi(x)\sigma((x,\infty),i).
\end{align*}
For $\Delta_2$ it holds
\begin{align*}
&\Delta_2=
\sum_{k\not= i} \int_{0}^{+\infty}\frac{\frac 1{c_r}\varphi(y)\varphi(x)}{\frac 1{c_r}\varphi(y)} \sigma(dy,k) =\varphi(x)\sum_{k\not= i}\sigma((0,\infty),k),
\end{align*}
and
 for $\Delta_3$
\begin{align*}
&\Delta_3=\frac{\frac 1{c_r}\varphi(x)}{\frac 1{c_r}}\sigma(\0)=\varphi(x)\sigma(\0).
\end{align*}
Consequently, since $\sigma$ is a probability measure, we have 
\begin{equation}
\label{fu1}
 f(x,i)=\Delta_1+\Delta_2+\Delta_3=\frac{c_r}{p_i}\varphi(x)v(x,i)+\varphi(x),
\end{equation}
where 
\begin{equation}
\label{v1}
v(x,i):= \int_{(0,x]}\frac{\psi^\partial(y)}{\varphi(y)}\,  \sigma(dy,i)+\frac{\psi^\partial(x)}{\varphi(x)}\sigma((x,\infty),i)
\end{equation}
We prove next that $f^+(\0,i)$ exists and find its value. For this it is enough to calculate 
$$
v^+(0,i):=\lim_{\delta \downarrow 0} 
\frac{v(\delta,i)-v(0,i)}{S(\delta)-S(0)}=\lim_{\delta \downarrow 0}  \frac{v(\delta,i)}{S(\delta)},
$$
where we used $v(0,i)=0$ and $S(0)=0$. 
Consider now the limit of  the first term on the right hand side of (\ref{v1}):
$$
0\leq \frac1{S(\delta)} \int_{(0,\delta]}\frac{\psi^\partial(y)}{\varphi(y)}\,  \sigma(dy,i)\leq 
\frac1{S(\delta)}\frac{\psi^\partial(\delta)}{\varphi(\delta)}\sigma((0,\delta],i)\to 0,
$$
as $\delta\to 0$, since (cf. (\ref{w0}))
$$
\frac1{S(\delta)}\frac{\psi^\partial(\delta)}{\varphi(\delta)}\to \frac{w_r}{\varphi(0)^2}=1,
$$
and $\sigma((0,\delta),i)\to 0$. Hence, $v^+(0,i)=\sigma((0,\infty),i)$ and we have (cf. (\ref{cr1}))
\begin{equation}
\label{f+1}
f^+(\0,i)=\frac{c_r}{p_i} \sigma((0,\infty),i)+\frac{d\varphi}{dS}(0)=\frac{c_r}{p_i} \sigma((0,\infty),i)-c_r.
\end{equation}
Consequently,
\begin{align}
\label{D00}
&\nonumber\mathbf{D} {f}(\0)=\sum_{i=1}^n p_i f^+(\0,i)=\sum_{i=1}^n\lp p_i \lp\frac{c_r}{p_i} \sigma((0,\infty),i)-c_r\rp\rp\\
&\hskip1cm\nonumber
=c_r\sum_{i=1}^n\lp\sigma((0,\infty),i)-p_i\rp\\
&\hskip1cm\nonumber
=-c_r\lp1- \sum_{i=1}^n \sigma((0,\infty),i) \rp \\
&\hskip1cm
=-c_r\, \sigma(\{\0\})\leq 0, 
\end{align}
where it is used that $\sigma $ is a probability measure. This proves (\ref{glue1}). 
Statement (\ref{glue2})  can be verified with similar calculations. We skip the details. 
\end{proof}


In the next two theorems, the first one with some overlapping with Theorem \ref{glue}, we give explicit forms of the representation measure of a given excessive function.  These formulas are  our key tools in solving optimal stopping problems for diffusion spiders. In the first theorem we consider excessive functions taking value 1 at the origin $\0$. In the second theorem general formulas are presented  when the value 1 is attained at some other point on $\Gamma$.    

\relax
\begin{theorem}
\label{rep0}
Let $f$ be a non-trivial $r$-excessive function of $\X$ such that $f(\0)=1$, and $\sigma$ its representing measure. Then 
for all $x\geq 0 $ and $i=1,2,...,n$
\begin{equation}\label{rep01}
\sigma((x,\infty),i)= \frac{p_i}{c_r}\lp f^+(x,i)\varphi(x)-\varphi^+(x)f(x,i)\rp,
\end{equation}
\begin{equation}\label{rep03}
\sigma(\{x\},i) =\frac{p_i}{c_r\varphi(x)}\left( f^-(x,i) - f^+(x,i)\right),
\end{equation}

and

\begin{equation}\label{rep02}
\sigma(\{\0\})=-\frac1{c_r}\sum_{i=1}^n p_i f^+(\0,i)=-\frac1{c_r}{\bf D}f(\0).
\end{equation}

\end{theorem}
\relax
\begin{proof}
Recall that $f$ has the representation (\ref{fu1}). Our first task is, hence, to find the derivative of $v$ defined as   
$$
v^+(x,i):=\lim_{\delta \downarrow 0}  \frac{v(x+\delta,i)-v(x,i)}{S(x+\delta)-S(x)}.
$$
It holds for $\delta\geq 0$
\begin{align*}
{v(x+\delta,i)-v(x,i)} = \int_{(x,x+\delta]}\frac{\psi^\partial(y)}{\varphi(y)}\,  \sigma(dy,i)+\frac{\psi^\partial(x+\delta)}{\varphi(x+\delta)}\sigma((x+\delta,\infty),i) -\frac{\psi^\partial(x)}{\varphi(x)}\sigma((x,\infty),i).
\end{align*}
Since
$$
\lim_{\delta \downarrow 0} \frac{1}{S(x+\delta)-S(x)}\,\lp \int_{(x,x+\delta]}\frac{\psi^\partial(y)}{\varphi(y)}\,  \sigma(dy,i)-\frac{\psi^\partial(x)}{\varphi(x)}\sigma((x,x+\delta],i)\rp=0,
$$
we have 
$$
v^+(x,i)=\lim_{\delta \downarrow 0}  \frac{1}{S(x+\delta)-S(x)}\lp \frac{\psi^\partial(x+\delta)}{\varphi(x+\delta)} - \frac{\psi^\partial(x)}{\varphi(x)}\rp \sigma((x+\delta,\infty),i)=\frac{1}{(\varphi(x))^2}\sigma((x,\infty),i),
$$
where (\ref{w0}) is used. Consequently, cf. (\ref{fu1}),
\begin{align*}
&f^+(x,i)=\frac{c_r}{p_i}\varphi^+(x)v(x,i)+\frac{c_r}{p_i}\varphi(x)v^+(x,i)+\varphi^+(x)\\
&\hskip1cm=\frac{c_r}{p_i}\varphi^+(x)v(x,i)+\frac{c_r}{p_i}\frac{\sigma((x,\infty),i)}{\varphi(x)}+\varphi^+(x), 
\end{align*}
and (\ref{rep01}) results after a straightforward computation. Recalling that $\varphi(0)=1,$ $c_r=-\varphi^+(0),$ and  $f(\0)=1$   formula (\ref{rep02}) is obtained from (\ref{rep01}) when taking $x=0$, also cf. (\ref{f+1}) and (\ref{D00}). 
\end{proof}

We consider now an $r$-excessive function $f$  taking value 1 at a point $\x_o=(x_o,i_o)\in\Gamma\setminus\{\0\}$. Then there exists a probability measure $\sigma_{x_o}$ such that for all $(x,i)\in\Gamma$ it holds 
\begin{equation}
\label{ma1}
    f(x,i) = \sum_{k=1}^n \int_{0}^{+\infty}\frac{\g_r((x,i),(y,k))}{\g_r(\x_o,(y,k))}\,\sigma_{x_o}(dy,k) + 
    \frac{\g_r((x,i),\0)}{\g_r(\x_o,\0)}\,\sigma_{x_o}(\{\0\}).
\end{equation}

\relax

\begin{theorem}
\label{rep1}
Let $f$ be an $r$-excessive function with the representation as in (\ref{ma1}).  Then 
for  $x\geq x_o $ on the leg $i_o$ 
\begin{equation}\label{rep11}
\sigma_{\x_o}((x,\infty),i_o)= p_{i_o} \tilde\psi(\x_o)\lp f^+(x,i_o)\varphi(x)-\varphi^+(x)f(x,i_o)\rp,
\end{equation}
and for  $0\leq x\leq x_o $
\begin{align}\label{rep12}
\nonumber
&\sigma_{\x_o}((0,x),i_o) + \sigma_{\x_o}(\{\0\})+\sum_{k\not= i_0}\sigma_{\x_o}((0,\infty),k)\\
&\hskip3cm= 
p_{i_o} \varphi(x_o)\lp f(x,i_o) \tilde\psi^-(x,i_o)-f^-(x,i_o) \tilde\psi(x,i_o)\rp.
\end{align}

\end{theorem}
\relax
\begin{proof}
Let $x\geq x_o $ on the leg $i_o$. Then 
\begin{align*}
 f(x,i_0) &= \int_{0}^{+\infty}\frac{\g_r((x,i_0),(y,i_0))}{\g_r(\x_o,(y,i_0))}\,\sigma_{\x_o}(dy,i_0) + \sum_{k \neq i_0} \int_{0}^{+\infty}\frac{\g_r((x,i),(y,k))}{\g_r(\x_o,(y,k))}\,\sigma_{\x_o}(dy,k) \\ &+ \frac{\g_r((x,i_0),\0)}{\g_r(\x_o,\0)}\,\sigma_{\x_o}(\{\0\}) := \Delta_1 + \Delta_2 + \Delta_3.
\end{align*}
Using \eqref{gk1}, one readily verifies that
\begin{align*}
\Delta_1 &= \frac{\varphi(x)}{\varphi(x_0)} \sigma_{x_0}((0,x_0],i_0) + \int_{(x_0,x]} \frac{\tilde{\psi}(y,i_0)\varphi(x)}{\tilde{\psi}(x_0,i_0)\varphi(y)} \sigma_{\x_0}(dy,i_0) + \frac{\tilde{\psi}(x,i_0)}{\tilde{\psi}(x_0,i_0)}\sigma_{\x_0}((x,\infty),i_0),\\
\Delta_2 &= \sum_{k \neq i_0} \frac{\varphi(x)}{\varphi(x_0)}\sigma_{\x_0}((0,\infty),k),\\
\Delta_3 &= \frac{\varphi(x)}{\varphi(x_0)}\sigma_{\x_0}(\{ \mathbf{0} \}).
\end{align*}
Since $\sigma_{x_0}$ is a probability measure, a simplification yields
\begin{align*}
f(x,i_0) &= \frac{\varphi(x)}{\varphi(x_0)}\sigma_{\x_0}((0,x_0],i_0) + \int_{(x_0,x]} \frac{\tilde{\psi}(y,i_0)\varphi(x)}{\tilde{\psi}(x_0,i_0)\varphi(y)} \sigma_{\x_0}(dy,i_0) + \frac{\tilde{\psi}(x,i_0)}{\tilde{\psi}(x_0,i_0)}\sigma_{\x_0}((x,\infty),i_0) \\ &+
\frac{\varphi(x)}{\varphi(x_0)}(1-\sigma_{\x_0}((0,\infty),i_0)) \\
&= \varphi(x)\left( \int_{(x_0,x]} \frac{\tilde{\psi}(y,i_0)}{\tilde{\psi}(x_0,i_0)\varphi(y)} \sigma_{\x_0}(dy,i_0) + \frac{\tilde{\psi}(x,i_0)}{\varphi(x)\tilde{\psi}(x_0,i_0)}\sigma_{\x_0}((x,\infty),i_0) \right) \\ &+ \varphi(x)\frac{1-\sigma_{\x_0}((x_0,\infty),i_0)}{\varphi(x_0)}\\
& := \varphi(x) v(x,i_0) + \varphi(x)\frac{1-\sigma_{\x_0}((x_0,\infty),i_0)}{\varphi(x_0)}. 
\end{align*}
The aim is to calculate the derivative of $f.$ For this, consider
\begin{align*}
v(x+\delta,i_0) - v(x,i_0) &= \int_{(x,x+\delta]} \frac{\tilde{\psi}(y,i_0)}{\tilde{\psi}(x_0,i_0)\varphi(y)} \sigma_{\x_0}(dy,i_0) + \frac{\tilde{\psi}(x+\delta,i_0)}{\varphi(x+\delta)\tilde{\psi}(x_0,i_0)}\sigma_{\x_0}((x+\delta,\infty),i_0) \\ &- \frac{\tilde{\psi}(x,i_0)}{\varphi(x)\tilde{\psi}(x_0,i_0)}\sigma_{\x_0}((x,\infty),i_0).
\end{align*}
This yields
\begin{align*}
v^+(x,i_0) &= \lim_{\delta\downarrow0} \frac{1}{S(x+\delta)-S(x)}\left(\frac{\tilde{\psi}(x+\delta,i_0)}{\varphi(x+\delta)} - \frac{\tilde{\psi}(x,i_0)}{\varphi(x)}\right)\frac{\sigma_{\x_0}((x+\delta,\infty),i_0)}{\tilde{\psi}(x_0,i_0)} \\
&= \lim_{\delta\downarrow0} \frac{1}{S(x+\delta)-S(x)}\left(\frac{\psi^\partial(x+\delta)}{\varphi(x+\delta)} - \frac{\psi^\partial(x)}{\varphi(x)}\right)\frac{\sigma_{\x_0}((x+\delta,\infty),i_0)}{p_i\tilde{\psi}(x_0,i_0)} \\
&= \frac{1}{p_i\varphi^2(x)}\frac{\sigma_{\x_0}((x,\infty),i_0)}{\tilde{\psi}(x_0,i_0)}.
\end{align*}
Finally, we obtain by substituting and differentiating the following expression
\begin{align*}
f^+(x,i_0) &= \varphi^+(x)v(x,i_0) + \varphi(x) \frac{1}{p_i\varphi^2(x)}\frac{\sigma_{\x_0}((x,\infty),i_0)}{\tilde{\psi}(x_0,i_0)} + \varphi^+(x)\frac{1-\sigma_{\x_0}((x_0,\infty),i_0)}{\varphi(x_0)}, 
\end{align*}
which implies \eqref{rep11}. The expression \eqref{rep12} is proved similarly, we omit the details.
\end{proof}


\relax 
We can also characterize $r$-excessive functions in the following way (cf. 
\cite{BayZhang},
\cite{FitzsimmonsKuter}, 
and \cite{KaratzasYan}).

\begin{theorem}
\label{charc2}
A measurable function $h\colon\Gamma\to[0,\infty)$ is $r$-excessive for $\X$ if and only if
\begin{itemize}
\item[\rm(a)] the function $x\mapsto h(x,i)$, $x>0, i=1,2,\dots,n$,  is $r$-excessive for $X^\partial$,  
\item[\rm(b)] the gluing condition holds:  $$\displaystyle{{\bf D}{h}(\0):=\sum_{i=1}^n p_ih^+(\0,i)}\leq 0.$$ 
\end{itemize}
\end{theorem}
\relax
\begin{proof}
Assume first  that $h$ is $r$-excessive for $\X$. Then (b) holds by (\ref{glue1}) in Theorem \ref{glue}. To verify (a), we recall from (\ref{fu1}) and (\ref{v1}) that on the leg $i$
\begin{align*}
&h(x,i)=\frac{c_r}{p_i}\varphi(x)v(x,i)+\varphi(x)\\
&\hskip1cm
=
\frac{c_r}{p_i} \left(\int_{(0,x]}\frac{\varphi(x)\psi^\partial(y)}{\varphi(y)}\,  \sigma(dy,i)+\psi^\partial(x)\sigma((x,\infty),i)\right) +\varphi(x).
\end{align*}
Using herein the explicit form of the Green function $g^\partial_r$ associated with $X^\partial$ as given in (\ref{gr0}) yields 
\begin{equation}\label{exh}
h(x,i)=\frac{c_r}{p_i}\int_0^\infty g_r^\partial(x,y)\, \mu(dy,i)+\varphi(x),
\end{equation}
where $\mu(dy,i):= \sigma(dy,i)/\varphi(y)$ is a finite measure on $\R_+$. Consequently, $h$ is $r$-excessive for $X^\partial$ since the first term on the right hand side of (\ref{exh}) is $r$-excessive by the Riesz representation and, clearly, also  the second term is such a one.

Assume next that (a) and (b) hold. We remark that (a) alone does not  imply the claim. Indeed, the function $h(x,i):=\psi^\partial(x) +1$ is continuous, strictly positive and $r$-excessive for $X^\partial$ but not $r$-excessive for $\X$. Indeed, if $h$ were $r$-excessive then the representing measure should put a non-negative mass at $\0$. But $h^+(\0,i)=1$ for very $i$ and, hence, from (\ref{rep02}), the mass is strictly negative, which contradicts the assumption of the $r$-excess.  
Since it is assumed that $x\mapsto h(x,i)$ is for every $i$ $r$-exessive for $X^\partial$ its representing probability measure is given  for every $i$ by
\begin{equation}\label{repar1}
\sigma^\partial((x,\infty],i)= 
C_1\lp h^+(x,i)\varphi(x)-\varphi^+(x)h(x,i)\rp,\qquad  x>x_o,
\end{equation}
and 
\begin{equation}\label{repar2}
\sigma^\partial(([0,x),i)= 
C_2\lp h(x,i) (\psi^\partial)^-(x,i)-h^-(x,i) \psi^\partial(x,i)\rp, \qquad 0<x<x_o,
\end{equation}
where $x_o>0$ is a point such that $h(x_o,i)=1$, and $C_1$ and $C_2$ are constants (which can be calculated explicitly and do not depend on $x$). 
Notice that (\ref{repar1}) and (\ref{repar2}) induce probability measures in fact for any $x_o>0$ since given an arbitray $x_o>0$ we may consider the representation of the function $x\mapsto h(x,i)/h(x_o,i)$. Consequently, the expression inside the parenthesis in (\ref{repar1}) is positive and decreasing as a function of $x>0$. Moreover, it is bounded by (b) and the assumptions on $\varphi$, cf. (\ref{norm}) and (\ref{reg00}).  We assume now without loss of generality that $h(\0)=1$. It is also assumed that $h$ has a representation as given on the right hand side of (\ref{ma}). The derivation of the formulas (\ref{rep01}) and  (\ref{rep02}) in Theorem \ref{rep0} is based only on the assumption that the representation (\ref{ma}) is valid. From the discussion above we conclude that the expression on the right hand side of  (\ref{rep01}) induces a measure. By the assumption (b) the mass at $\0$ is non-negative. Therefore, the function $h$ has the representation given on the right hand side of (\ref{ma}) with $\sigma$ as a probability measure. Consequently, $h$ is $r$-excessive for $\X$.  
\end{proof}

The proof of Theorem \ref{charc2} yields also the following characterization of excessive functions expressed solely in terms of $h$ and its derivatives. 

\relax

\begin{theorem}
\label{charc3}
A measurable function $h\colon\Gamma\to[0,\infty)$ is $r$-excessive for $\X$ if and only if
\begin{itemize}
\item[\rm(a)] for every $x>0$ and $i=1,2,\dots,n $ the derivatives $h^+(x,i)$ and $h^-(x,i)$ exist, 
\item[\rm(b)] for every $i=1,2,\dots,n$ the function
$$
 x\mapsto   h^+(x,i)\varphi(x)-\varphi^+(x)h(x,i),\quad x>0.
 $$
 is bounded, positive and decreasing,

\item[\rm(c)]  the gluing condition holds:   $$\displaystyle{D{h}(\0):=\sum_{i=1}^n p_ih^+(\0,i)}\leq 0.$$ 
\end{itemize}
\end{theorem}
\section{Optimal stopping}
\label{O}
\label{GE}
\bigskip
Consider a continuous function $g\colon \Gamma\to [0,\infty)$.
We assume moreover that  $g$ is differentiable (with respect to the scale) on the interior of the edges,
and has the right derivatives at $0$ along the edges.
%
Let  $\cM$ denote the set of stopping times adapted to the filtration generated by $\X$.
Our optimal stopping problem consists of finding a function $V$, the value function,  and a stopping time  $\tau^*$, an optimal stopping time,  such that 
\begin{equation}\label{eq:osp}
V(\x)=\sup_{\tau\in\cM}\EE_{\x}\left(e^{-r\tau}g(\X(\tau))\right)=\EE_{\x}\left(e^{-r\tau^*}g(\X(\tau^*))\right).
\end{equation}
In order for the  value function to be finite, we assume the following condition
to hold:
\begin{equation}\label{finite}
\EE_{\x}\lp\sup_{t\geq 0}(e^{-rt}g(\X(t)))\rp<\infty.
\end{equation}
In (\ref{eq:osp}), when $\tau(\omega)=+\infty$ we define
\begin{equation}\label{finite1}
e^{-r\tau}g(\X(\tau)):=\limsup_{t\to\infty}e^{-rt}g(\X(t)).
\end{equation}  
An alternative criterium to check  the finiteness of the value function is given in the next proposition making use of  the characterizations of harmonic functions in section \ref{EX}.
		\begin{proposition}\label{fin00}
			Assume that there exists a positive $r$-harmonic function $H$ of $\X$ such that  $g/H$ is bounded. Then, the value function in \eqref{eq:osp} is finite.
		\end{proposition}
	\relax
		\begin{proof} We use the argument presented, e.g., in \cite{CCMS} Proposition 3.3. p. 2570. For this, recall that $\{e^{-rt}H(\X(t))\}_{t\geq 0}$ is a positive martingale. Letting $C$ be a bound of $g/H$  then it holds  for an arbitrary stopping time $\tau$ 	
			\begin{align*}
				\EE_{\x}\left(e^{-r\tau}g(\X_{\tau})\right)  &=	\EE_{\x}
				\left(
				e^{-r\tau}{g(\X_{\tau})\over H(\X_{\tau})}H(\X_{\tau})
				\right)\\
				&\leq C \EE_{\x}\left(e^{-r\tau}H(\X_{\tau})\right)\\	
&=C H(\x)<\infty,
			\end{align*}
			where in case $\tau(\omega)=\infty$  
$$
e^{-r\tau}H(\X_{\tau}):=\lim_{t\to\infty} e^{-rt}H(\X_{t}).
$$
\end{proof}

\relax

\begin{theorem}
\label{st1}
The value function $V$ is the smallest $r$-excessive majorant of the reward function $g$. An optimal stopping  time of the problem (\ref{eq:osp}) is given by
$$
\tau^*=\inf\{t\geq 0\colon \X(t)\in \Sigma\},
$$ 
where 
\begin{equation}
\label{stop1}
\Sigma:=\{\x\in\Gamma\colon g(\x)=V(\x)\}.
\end{equation}
\end{theorem}
\relax
\begin{proof}
Observe that
the state space $\Gamma$ is a locally compact Hausdorff space with a
countable basis (i.e. semi compact), the process has continuous trajectories 
and is strong Markov. Furthermore the function $g$ is assumed to be continuous
and to satisfy the  condition \eqref{finite}. Then the first claim  follows from Theorem 1 p. 124 in \cite{Shiryaev} and the second one from Theorem 3 p. 127 in ibid.  
\end{proof}

The set $\Sigma$ introduced in (\ref{stop1})  is called the stopping region and its complement $C:=\Sigma^c$ the continuation region of the OSP (\ref{eq:osp}). 


The next result is  essentially Corollary on p. 124 in \cite{Shiryaev}, and can be used to verify that a candidate value function is indeed the true value (see, e.g.,  \cite{AlvarezSalminen}, \cite{MordeckiSalminen1}, and  \cite{MordeckiSalminen2}).  We apply  this, in particular, in our first example below. Recall that the underlying diffusion is assumed to be recurrent and, hence, the first hitting time $H_A$ defined below is a.s. finite.
\begin{theorem}\label{apu}
Let $A\subset \Gamma$ be a (non-empty) Borel subset of $\Gamma$ and 
$$
H_A:=\inf\{t\geq 0:\X_t\in A\}.
$$ 
Assume that the function 
$$
\hat{V}(\x):=\mathbb{E}_\x\left[\textrm{e}^{-rH_A}g(\X_{H_A})\right]
$$
 is $r$-excessive and dominates $g$. Then, $\hat{V}$ coincides with the value function of OSP (\ref{eq:osp}) and $H_A$ is an optimal stopping time.
\end{theorem}
\noindent
We have the following sufficient condition for $\0$ to belong to the continuation region. 

\begin{proposition}\label{COND}
If 
\begin{equation}\label{COND1}
{\bf D}{g}(\0):=\sum_{i=1}^np_i g^+(\0,i)>0,
\end{equation}
then $\0\in C$.
\end{proposition}
\relax
\begin{proof} Notice that ${\bf D}{g}(\0)$ is well defined by the assumptions on $g$. Suppose that (\ref{COND1}) holds.  If $\0\in\Sigma$ then, from (\ref{stop1}),  $g(0)=V(0)$. Since $V$ is $r$-excessive we have ${\bf D}{V}(\0)\leq 0$ by Theorem \ref{charc2}. But since $V(\x)\geq g(\x)$ for all $\x\in\Gamma$ it follows that 
$0<{\bf D}{g}(\0)\leq {\bf D}{V}(\0)\leq 0$ which is a contradiction. This concludes the proof of the proposition.
\end{proof}

\relax

We proceed now to characterize the stopping set as a unique solution to an integral equation. It is seen that these results can be derived by modifying/converting the approach presented in \cite{CCMS} for multidimensional diffusions to the present case with diffusion spiders. The first result is a verification theorem presenting the key integral equation for the stopping region and the second one is for the uniqueness of the solution of this equation. 
For this aim,   recall the notation, cf. (\ref{green11}) and (\ref{gk1}),
\begin{align}
\label{green00}
\nonumber
 \mathbf{G}_rf(x,i)&:=\EE_{(x,i)}\left[\int_0^{\infty}e^{-rt}\, f(\X_t)dt\right]\\
&= \sum_{k=1}^n \int_{0}^{+\infty}\g_r((x,i),(y,k))\,f(y,k)\,p_k\, m(dy),
\end{align}
with a measurable and bounded $f:\Gamma\mapsto \R$.  In fact, if $f$ is "only" locally bounded and satisfies
\begin{equation}
\label{ma3011}
\sum_{i=1}^n\int_0^\infty\varphi(y)|f(y,i)|\, p_i\,m(dy) <\infty
\end{equation} 
then we still have for all $(x,i)\in\Gamma$
\begin{equation*}
 |\mathbf{G}_rf(x,i)|<\infty.
\end{equation*}

\begin{theorem} \label{general} 
Assume there exist a locally bounded and measurable function $f:\Gamma\mapsto \R$ and a constant $\Delta_\0$ (both depending on the reward function $g$) such that (\ref{ma3011}) holds 
and  $g$ can be written  for all $(x,i)\in\Gamma$ as
\begin{equation}
\label{ma3}
    g(x,i) = \mathbf{G}_rf(x,i)
 +     \g_r((x,i),\0)\,\Delta_\0.
\end{equation}
Assume also that there exists a (non-empty) closed  Borel subset $O$ of $\Gamma$ such that 
\begin{description}
  	\item[(a)] if $(x,i)\in O$ then $f(x,i)\geq 0$,
	\item[(b)]  if $\Delta_\0< 0$ then $\0\in O^c:={\Gamma \setminus O}$,
\end{description}
 and that the function 
 ${\tilde V}\colon\Gamma \to \R$ defined by
\begin{equation}
\label{rie1}
\tilde V((x,i)):= \mathbf{G}_r({\bf 1}_Of)(x,i)
   +  \g_r((x,i),\0)\,{\bf 1}_O(\0)\,\Delta_\0
\end{equation}
satisfies the following two conditions:
\begin{description}
	\item[(1)] $\tilde V((x,i))\geq g((x,i))$ for all $(x,i)\in O^c$,
\item[(2)] $\tilde V((x,i))=g((x,i))$ for all $(x,i) \in \partial O$,  or, equivalently, for all $(x,i) \in \partial O$  
\begin{equation}\label{b0}
 g(x,i)-\tilde{V}(x,i)=\mathbf{G}_r({\bf 1}_{O^c}f)(x,i)
 + \g_r((x,i),\0)\,{\bf 1}_{O^c}(\0)\,\Delta_\0 =0. 
\end{equation}
\end{description}
Then, in the problem (\ref{eq:osp}) the stopping set $\Sigma=O$, 
$\tau^*=H_{O}:=\inf\{t\geq 0:\X_t\in O\}$ is an optimal stopping time
and the value function  $V=\tilde V$.  
\end{theorem}
\noindent
 The following general result is needed in the proof of Theorem \ref{general}.
\begin{lemma}
\label{apu1}
Let $f:\Gamma\mapsto \R$ be  a locally bounded and  measurable function satisfying (\ref{ma3011}) and $\tau$ a stopping time in the filtration generated by $\X$. Then
\begin{equation}
\label{Fyy1}
\mathbf{G}_rf(x,i)= \mathbb{E}_{(x,i)}\left[ e^{-r\tau}\mathbf{G}_rf(\X_\tau){\bf 1}_{\{\tau<\infty\}} + \int_0^\tau e^{-rt}f(\X_t)dt \right].
\end{equation}
\end{lemma}

\begin{proof}
The claim follows immediately from the strong Markov property. Indeed,
\begin{align}
\label{Dyy1}
\mathbf{G}_rf(x,i)
&= \mathbb{E}_{(x,i)}\Big[\left( \int_\tau^{\infty}e^{-rt}\, f(\X_t)dt + \int_0^\tau e^{-rt} f(\X_t)dt \right){\bf 1}_{\{\tau<\infty\}}\cr
&\hskip6cm+\int_0^{\infty}e^{-rt}\, f(\X_t)dt\,{\bf 1}_{\{\tau=\infty\}}\Big]\cr
& = \mathbb{E}_{(x,i)}\left[ e^{-r\tau} \mathbf{G}_rf(\X_\tau){\bf 1}_{\{\tau<\infty\}}+ \int_0^\tau e^{-rt} f(\X_t)dt \right]. 
\end{align}
\end{proof}
\noindent
{\sl Proof of Theorem \ref{general}.}
The proof consists of checking that the function  $\tilde{V}$ fulfills the conditions in Theorem  \ref{apu}. Firstly, notice that $ \tilde{V}$ is $r$-excessive since it is a linear combination of the Green kernel at $\0$ and an integral of the Green kernel with respect to a positive measure. Since, by assumption {\bf (1)},  $\tilde{V}$ dominates $g$ on $O^c$ it remains to show that 
for all  $(x,i) \in \Gamma$
\begin{equation}
\label{claim}
\tilde V(x,i)=\mathbb{E}_{(x,i)}\left[ e^{-rH_{O}} g(\X_{H_{O}}) \right].
\end{equation} 
Recall that $\X$ is recurrent and, hence, $H_{O}<\infty$ a.s. for all starting states. Moreover,  $\X_{H_{O}}\in \partial O\subseteq O$ since $O$ is assumed to be closed and $\X$ has continuous sample paths. Suppose now that $(x,i) \in O^c$ and $\0\in O^c$. Then, using (\ref{Fyy1}), 
\begin{align*}
\tilde V(x,i)=\mathbf{G}_r({\bf 1}_Of)(x,i)&= \mathbb{E}_{(x,i)}\left[ e^{-rH_{O}}\mathbf{G}_r({\bf 1}_Of)(\X_{H_{O}}) + \int_0^{H_{O}} e^{-rH_{O}}({\bf 1}_Of)(\X_t)dt \right]\\
&= \mathbb{E}_{(x,i)}\left[ e^{-rH_{O}}\mathbf{G}_r({\bf 1}_Of)(\X_{H_{O}}) \right]\\
&= \mathbb{E}_{(x,i)}\left[ e^{-rH_{O}}\tilde V(\X_{H_{O}})\right]\\
&= \mathbb{E}_{(x,i)}\left[ e^{-rH_{O}}g(\X_{H_{O}})\right],
\end{align*}
where, in the last step, assumption (2) is applicable since $\X_{H_O}\in \partial O$. If $\0\in O$
we have 
\begin{align}
\label{XO}
\nonumber
\tilde V(x,i)&= \mathbf{G}_r({\bf 1}_Of)(x,i) + \g_r((x,i),\0)\,\Delta_\0\\
&=\mathbb{E}_{(x,i)}\left[ e^{-rH_{O}}\mathbf{G}_r({\bf 1}_Of)(\X_{H_{O}}) \right]+    \g_r((x,i),\0)\,\Delta_\0.
\end{align}
To analyze the second term on RHS of (\ref{XO})  recall from  (\ref{gk1}) in Theorem \ref{Thrm:GreenKernel1} that  
$$
\mathbf{g}_r((x,i),\mathbf{0}) = c_r^{-1}\varphi(x).
$$
Hence, the function $\mathbf{g}_r(\cdot,\mathbf{0})$ is $r$-harmonic on the legs of $\Gamma$; in other words, if $A$ is a Borel subset of $[0,+\infty)$ including $0$, $\X_0=(x,i)\in (A^c,i)$ and $H_{A}:=\inf\{t\,:\, \X_t\in (A,i)\}$ then 
\begin{equation*}
\mathbf{g}_r((x,i),\mathbf{0})=\mathbb{E}_{(x,i)}\left[ e^{-rH_{{A}}}\mathbf{g}_r(\X_{H_{{A}}},\mathbf{0})\right].
\end{equation*}
Consequently, if $(x,i) \in O^c$ and $\0\in O$
 \begin{equation*}
\mathbf{g}_r((x,i),\mathbf{0})=\mathbb{E}_{(x,i)}\left[ e^{-rH_{{O}}}\mathbf{g}_r(\X_{H_{{O}}},\mathbf{0})\right],
\end{equation*} 
and (\ref{XO}) can be developed as follows 
\begin{align*}
\tilde V(x,i)&=\mathbb{E}_{(x,i)}\left[ e^{-rH_{O}}\mathbf{G}_r({\bf 1}_Of)(\X_{H_{O}}) \right]+  \mathbb{E}_{(x,i)}\left[ e^{-rH_{{O}}}\mathbf{g}_r(\X_{H_{{O}}},\mathbf{0})\right]\,\Delta_\0\\
&=\mathbb{E}_{(x,i)}\left[ e^{-rH_{O}}\Big(\mathbf{G}_r({\bf 1}_Of)(\X_{H_{O}}) +  \mathbf{g}_r(\X_{H_{{O}}},\mathbf{0})\,\Delta_\0\Big)\right]\\
&=\mathbb{E}_{(x,i)}\left[ e^{-rH_{O}}\,\tilde V(\X_{H_{{O}}})\right]\\
&=\mathbb{E}_{(x,i)}\left[ e^{-rH_{O}}\,g(\X_{H_{{O}}})\right]
\end{align*}
proving (\ref{claim}) also in this case. To show that (\ref{claim}) holds when $(x,i)\in O$ we show that $\tilde V=g$ on $O$. This clearly implies the claim since if   $(x,i)\in O$ then a.s. $H_O=0$ in (\ref{claim}). Subtracting (\ref{rie1}) from (\ref{ma3}) yields
 for all $(x,i)$
\begin{align}
\label{Fyy2}
g(x,i) - \tilde{V}(x,i)&=\mathbf{G}_r({\bf 1}_{O^c}f)(x,i)
 + \g_r((x,i),\0)\,{\bf 1}_{O^c}(\0)\,\Delta_\0
\end{align}
Let $(x,i) \in O$ and $\0\in O$. Then, evoking (\ref{Fyy1}),
\begin{align*}
g(x,i) - \tilde{V}(x,i)&=\mathbf{G}_r({\bf 1}_{O^c}f)(x,i)\\
&= \mathbb{E}_{(x,i)}\left[ e^{-rH_{O^c}}\mathbf{G}_r({\bf 1}_{O^c}f)(\X_{H_{O^c}}) + \int_0^{H_{O^c}} e^{-rH_{O^c}}({\bf 1}_{O^c}f)(\X_t)dt \right]\\
&= \mathbb{E}_{(x,i)}\left[ e^{-rH_{O^c}}\mathbf{G}_r({\bf 1}_{O^c}f)(\X_{H_{O^c}}) \right]\\
&=0,
\end{align*}
where in the in the last step  assumption 2 is applicable since $\X_{H_{O^c}}\in \partial O$. Finally, suppose $(x,i) \in O$ and $\0\in O^c$. Then, with analogous arguments as above,
\begin{align*}
g(x,i) - \tilde{V}(x,i)&=\mathbf{G}_r({\bf 1}_{O^c}f)(x,i)+ \mathbf{g}_r((x,i),\mathbf{0})\Delta_{\mathbf{0}}\\
&= \mathbb{E}_{(x,i)}\left[ e^{-rH_{O^c}}\mathbf{G}_r({\bf 1}_{O^c}f)(\X_{H_{O^c}}) \right]+\mathbb{E}_{(x,i)}\left[ e^{-rH_{{O}^c}}\mathbf{g}_r(\X_{H_{{O}^c}},\mathbf{0})\Delta_{\mathbf{0}}\right]
\\
&= \mathbb{E}_{(x,i)}\left[ e^{-rH_{O^c}}\Big(\mathbf{G}_r({\bf 1}_{O^c}f)(\X_{H_{O^c}}) +\mathbf{g}_r(\X_{H_{{O}^c}},\mathbf{0})\Delta_{\mathbf{0}}\Big)\right]\\
&=0.
\end{align*}
This completes the proof.$\hskip10.7cm\square$

\relax

\relax

\begin{theorem}\label{uniqueness_green}
Let $g, f, O, \Delta_\0$ and $\tilde V$ satisfy the assumptions in Theorem \ref{general}. 
Furthermore, let $\cal U$ be a class of nonempty closed subsets of
$\Gamma$ containing $O$ such that for all $U, U'\in{\cal U}$
\begin{description}
\item[(a)]$U\cap U'\not= \emptyset$,
		\item[(b)]$(x,i)\in U \Rightarrow f(x,i)\geq 0$, 
\item[(c)] $\0\in U^c$,  
	\item[(d)] the following implications  hold true:
	\begin{description}
		\item[(i)] if $\P_{(x,i)}({\rm Leb}(\{t\leq H_{U^c}:\X_t\not\in U'\})=0)=1$ for all $(x,i)\in U\cap U'$, then $U\subseteq U'$,
		\item[(ii)] if $\P_{(x,i)}({\rm Leb}(\{t\leq H_{U}\:\X_t\in U'\})=0)=1$ for all $(x,i)\in U^c$, then $U'\subseteq U$,
	\end{description}
where ${\rm Leb}$ denotes the Lebesgue measure.
\end{description}
Then $O$ is the unique set in the class $\cal U$ which satisfies the equation (\ref{b0}).
\end{theorem}

\begin{proof}
Let $U \in \mathcal{U}$ be such that for all $(x,i) \in \partial U$
\begin{align}\label{00}
{\bf G_r}(f\mathbf{1}_{U^c})(x,i) + \mathbf{g}_r((x,i),\mathbf{0})\mathbf{1}_{U^c}(\mathbf{0})\Delta_{\mathbf{0}}=0,
\end{align}
i.e., (\ref{b0}) holds when $O$ is replaced with $U$. The claim is that $U=O$. To prove this, define  
\begin{align*}
W(x,i) := {\bf G_r}(f\mathbf{1}_{U})(x,i) + \mathbf{g}_r((x,i),\mathbf{0})\mathbf{1}_{U}(\mathbf{0})\Delta_{\mathbf{0}}.
\end{align*}
Using  the assumed representation of the reward function $g$, see \eqref{ma3}, yields 
\begin{align*}
g(x,i)-W(x,i) &= {\bf G_r}(f\mathbf{1}_{U^c})(x,i) + \mathbf{g}_r((x,i),\mathbf{0})\mathbf{1}_{U^c}(\mathbf{0})\Delta_{\mathbf{0}}.
\end{align*}
From (\ref{00}) it follows that for all $(x,i) \in \partial U$ 
\begin{align}\label{000}
g(x,i)-W(x,i)=0,
\end{align}
and arguing analogously as in the proof of Theorem \ref{general} it is seen that (\ref{000}) holds for all $(x,i) \in U$. Recall that $\tilde V$ is the value function of the problem. We show now that for all $(x,i) \in \Gamma$
\begin{align}\label{Eq:UniquenessProofInEq1}
W(x,i) \leq\tilde V(x,i). 
\end{align}
Firstly, for $(x,i) \in U$ it holds $W(x,i)=g(x,i)\leq V(x,i)$ since $\tilde V$ is a majorant of $g$. Secondly, let $(x,i) \in U^c$. Proceeding similarly as when proving (\ref{claim}) in the proof of Theorem \ref{general} it is seen that
\begin{align*}
W(x,i) &=\mathbb{E}_{(x,i)}\left[e^{-rH_{U}} W(\X_{H_{U}})\right] \\
&= \mathbb{E}_{(x,i)}\left[e^{-rH_{U}} g(\X_{H_{U}})  \right] \\
&\leq  \tilde V(x,i),
\end{align*}
where the second equality is valid since  $\X_{H_{U}}\in\partial U$, and the inequality holds  since $\tilde V$ is the value of the problem. Hence, (\ref{Eq:UniquenessProofInEq1}) is verified.

Next we show that  $O \subseteq U$. Recall the assumption  {\bf (a)}, and let  $(x,i) \in O \cap U$. Using Lemma \ref{apu1}, the fact that the function $\mathbf{g}_r(\cdot,\mathbf{0})$ is $r$-harmonic on the legs of $\Gamma$, and the inequality $\eqref{Eq:UniquenessProofInEq1}$, yield
\begin{align}
\label{ff1}
\nonumber
g(x,i) &= W(x,i)\\
\nonumber
&= \mathbb{E}_{(x,i)}\left[ e^{-rH_{O^c}}\mathbf{G}_r(f{\bf 1}_{U})(\X_{H_{O^c}}) \right]+\mathbb{E}_{(x,i)}\left[ e^{-rH_{{O}^c}}\mathbf{g}_r(\X_{H_{{O}^c}},\mathbf{0})\mathbf{1}_{U}(\mathbf{0})\Delta_{\mathbf{0}}\right]\\
\nonumber
                           &\hskip2cm+ \mathbb{E}_{(x,i)}\left[ \int_0^{H_{O^c}} e^{-rs} f(\X_s)\mathbf{1}_U(\X_s)ds \right]  
\\
\nonumber
				&= \mathbb{E}_{(x,i)}\left[ e^{-rH_{O^c}} W(\X_{H_{O^c}}) \right] + \mathbb{E}_{(x,i)}\left[ \int_0^{H_{O^c}} e^{-rs} f(\X_s)\mathbf{1}_U(\X_s)ds \right] \\
				&\leq \mathbb{E}_{(x,i)}\left[ e^{-rH_{O^c}} \tilde V(\X_{H_{O^c}}) \right] 
                          + \mathbb{E}_{(x,i)}\left[ \int_0^{H_{O^c}} e^{-rs} f(\X_s)\mathbf{1}_U(\X_s)ds \right].
\end{align}
On the other hand, since $\0\in O^c$ by assumption {\bf (c)} we have (see also  (\ref{rie1}))   
\begin{align}
\label{ss1}
\nonumber
g(x,i) &= \tilde V(x,i) \\	\nonumber
\nonumber
&=\mathbf{G}_r(f{\bf 1}_{O})(x,i) +   \g_r((x,i),\0)\,{\bf 1}_O(\0)\,\Delta_\0\\
\nonumber
&=\mathbf{G}_r(f{\bf 1}_{O})(x,i) \\
\nonumber
&= \mathbb{E}_{(x,i)}\left[ e^{-rH_{O^c}}\mathbf{G}_r(f{\bf 1}_{O})(\X_{H_{O^c}}) \right]+ \mathbb{E}_{(x,i)}\left[ \int_0^{H_{O^c}} e^{-rs} f(\X_s)\mathbf{1}_O(\X_s)ds \right] 
\\
	&=\mathbb{E}_{(x,i)}\left[ e^{-rH_{O^c}} V(\X_{H_{O^c}}) \right] 
                          + \mathbb{E}_{(x,i)}\left[ \int_0^{H_{O^c}} e^{-rs} f(\X_s)ds \right].
\end{align}
Subtracting (\ref{ss1}) from (\ref{ff1}) yields
\begin{align*}
0 &\leq   \mathbb{E}_{(x,i)}\left[ \int_0^{H_{O^c}} e^{-rs} f(\X_s)\mathbf{1}_U(\X_s)ds \right]-\mathbb{E}_{(x,i)}\left[ \int_0^{H_{O^c}} e^{-rs} f(\X_s)ds \right]   \\
	&= - \mathbb{E}_{(x,i)}\left[ \int_0^{H_{O^c}} e^{-rs} f(\X_s)\mathbf{1}_{U^c}(\X_s)ds \right] \leq 0.
\end{align*}
Thus $\mathbb{P}_{(x,i)}({\rm Leb}(\{t \leq {H_{O^c}} | \X_t \in U^c \})=0)=1$ for all $(x,i)\in O \cap U$,  and  it follows from assumption {\bf c(i)} that $O \subseteq U$. 

Finally, we prove the opposite inclusion, i.e., $U\subseteq O$. For this,  let $(x,i)\in O^c$ and recall  that  $H_O=\inf\{ t\geq0 | X_t \in O \}$ is an optimal stopping time.  Calculating similarly as above we have 
\begin{align*}
V(x,i) &\geq W(x,i) \\
&= {\bf G_r}(f\mathbf{1}_{U})(x,i) + \mathbf{g}_r((x,i),\mathbf{0})\mathbf{1}_{U}(\mathbf{0})\Delta_{\mathbf{0}}\\
&= {\bf G_r}(f\mathbf{1}_{U})(x,i) \\
&= \mathbb{E}_{(x,i)}\left[ e^{-rH_{O}}\mathbf{G}_r(f{\bf 1}_{U})(\X_{H_{O}}) \right]+ \mathbb{E}_{(x,i)}\left[ \int_0^{H_{O}} e^{-rs} f(\X_s)\mathbf{1}_U(\X_s)ds \right] 
\\		
&= \mathbb{E}_{(x,i)}\left[ e^{-r H_O} W(\X_{H_O}) \right] + \mathbb{E}_{(x,i)}\left[ \int_0^{H_O} e^{-rs}f(\mathbf{\X}_s)\mathbf{1}_{U}(\mathbf{\X}_s)ds \right].
\end{align*}
Since $O \subseteq U$ and $W=g$ on $U$, it holds that $W(\mathbf{X}_{H_O})=g(\mathbf{X}_{H_O})$. Thus
\begin{align*}
V(x,i) \geq  \mathbb{E}_{(x,i)}\left[ e^{-r H_O} g(X_{H_O}) \right] + \mathbb{E}_{(x,i)}\left[ \int_0^{H_O} e^{-rs}f(\mathbf{X}_s)\mathbf{1}_{U}(\mathbf{X}_s)ds \right] \geq V(x,i). 
\end{align*}
This implies $\mathbb{P}_{(x,i)}({\rm Leb}(\{t \leq H_O | \X_t \in U \})=0)=1$ for all $(x,i)\in O^c$. Now by assumption {\bf d(ii)} we conclude $U \subseteq O$.
\end{proof}

\relax

\relax

To apply Theorem \ref{general} it is helpful to express $f$ and $\Delta_\0$  in terms of the function $g$ and its derivatives. This is achieved with similar analysis as presented  for excessive functions in Theorem~\ref{rep0} and the proof is Appendix A. 

\begin{proposition}\label{rwdmeasure}
Assume that the reward function $g$  has the representation (\ref{ma3}). Then 
\begin{description}
	\item[(a)] $\Delta_\0=-{\bf D}{g}(\0)$,
	\item[(b)] for all $x\geq 0$ and $i=1,2,...,n$
\begin{equation}
\label{ms}
	\int_{(x,\infty)} \varphi(y)\,f(y,i)\,m(dy)= g^+(x,i)\varphi(x)  - g(x,i)\varphi^+(x).
	\end{equation} 
\end{description}
\end{proposition}
 
\begin{remark}
\label{rm11}
{\sl 
As is seen from the proof, the existence of the derivatives of $g$, in fact, follows from the assumption that $g$ satisfies (\ref{ma3}). Notice also that if the speed measure $m$ has a continuous derivative $m'$ with respect to the Lebesgue measure, the scale function $S$ is continuously differentiable and $g$ is two times continuously differentiable then  the formula (\ref{ms}) yields 
\begin{equation}
\label{msx}
f(x,i) =- \lp \frac d{dm}\frac d{dS}g(x,i)-rg(x,i)\rp{\color{black}= \lp r-\cG\rp g(x,i)}.
\end{equation} 
Hence, (\ref{ma3}) expresses, in a sense, the classical  fact  that the infinitesimal operator is the inverse of the resolvent operator. However, to make this precise, $g$ should be in the domain of the infinitesimal operator. We refer to \cite{CM14,CrocceMordecki} for this approach for one-dimensional diffusions. Instead of giving conditions for the reverse of Proposition \ref{rwdmeasure} to hold we calculate in the examples below $\Delta_\0$ as given in (a) and $f$ as given in (\ref{msx}) and verify that (\ref{ma3}) is satisfied.} \end{remark}
\relax

\relax

\section{Examples of optimal stopping}
\label{OSP}
\label{EXA1}
In our examples to follow the underlying diffusion $\X$ is a Brownian spider with $n$ legs.
In Example \ref{EXA11} the stopping region is connected and the problem is solved using the explicit expressions  for the representing measure of an excessive function displayed in Theorem \ref{rep0} and \ref{rep1}. Here it is also studied how the stopping region changes when the discounting parameter $r$ is varying.  
In Examples \ref{ex2} and \ref{ex3} the stopping region is disconnected. 
In this cases, we analyze the equation (\ref{b0}) in Theorem \ref{general} and deduce using Theorem \ref{uniqueness_green} that there is a unique solution. 
The resulting equations for the boundary points of the stopping region  are then solved numerically. 
\subsection{Example}
\label{EXA11}

Consider the optimal stopping problem (\ref{eq:osp}) for the Brownian spider $\X$ with 3 legs and the following payoff function
$$
g(x,1)=1+x,\qquad g(x,2)=\lp 1-\frac x2\rp^+,\qquad g(x,3)=(1-2x)^+.
$$
It is also assumed that   $p_1=p_2=p_3=1/3$. This assumption slightly simplifies the analysis but does not  diminish the understanding of more general cases.
That the problem has a finite value follows from Proposition \ref{fin00} 
evoking the explicit form of the harmonic function given in (\ref{mahar23}). 
The solution of the problem is given in the next proposition.  It is seen that the stopping region is connected for all values of the discounting parameter $r.$
\begin{proposition}\label{solu}
The stopping region $\Sigma$ of OSP  introduced above is as follows:
\begin{itemize}
\item[\rm(a)] in case $r>2$ there exist $x^*_2\in(0,2)$ and $x^*_3\in(0,1/2)$ such that 
\begin{equation}\label{S1}
\Sigma=\{(x,1) : 0\leq x\}\cup\{(x,2) : 0\leq x\leq x^*_2\}\cup\{(x,3) : 0\leq x\leq x^*_3\},
\end{equation}
\item[\rm(b)] in case $1/8<r\leq 2$  there exists $y^*_2\in(0,2)$ such that 
\begin{equation}\label{S2}
\Sigma=\{(x,1) : 0\leq x\}\cup\{(x,2) : 0\leq x\leq y^*_2\},
\end{equation}
\item[\rm(c)]  in case $0<r\leq 1/8$  there exists $z^*_1\geq 0$ such that 
\begin{equation}\label{S3}
\Sigma= \{(x,1) : z^*_1 \leq x\},
\end{equation}
\end{itemize}
 where any point $(0,i), i=1,2,3,$ is identified as the origin $\0$. 
\end{proposition}
\relax
\begin{proof}
The idea of the proof is to study the expressions (\ref{rep01}) and (\ref{rep02})  for $g$ and deduce what is the largest set such that the induced measure is a probability measure. Firstly, recalling that the scale function is $S(x)=x$ we have 
\begin{equation}\label{D1}
\mathbf{D} g(\0)= \frac 1 3 \lp g'(0,1)+g'(0,2)+g'(0,3)\rp= \frac 1 3 \lp 1-\frac 1 2- 2\rp=-\frac 1 2.
\end{equation}
This means that the payoff function satisfies the condition (\ref{glue1}) which should hold for excessive functions at $\0$ and, hence, it is possible that $\0\in \Sigma.$ Notice that if  $\0\in \Sigma$ then $V(\0)=g(\0)=1$. 
Consider next the function on the right hand side of (\ref{rep01}) for leg 2 evaluated for $f=g,$ that is, for $x\in(0,2),$ $i=2$, and $\theta:=\sqrt{2r}$
\begin{equation}
\label{f2}
  \frac{p_2}{c_r}\lp g^+(x,2)\varphi(x)-\varphi^+(x)g(x,2)\rp=\frac1{3\theta}\, \e^{-\theta x}\lp-\frac 12+\theta\lp1-\frac x 2\rp
  \rp =: F_2(x)
\end{equation}
Clearly, $F_2$ is decreasing and $F_2(2)<0$. Hence, if $F_2(0)>0$   there is a unique zero. This is the case when $r>1/8.$    
We study similarly $F_3$ defined for $x\in(0,1/2)$ by 
$$
F_3(x):=  \frac{p_3}{c_r}\lp g'(x,3)\varphi(x)-\varphi^+(x)g(x,3)\rp=\frac1{3\theta}\, \e^{-\theta x}\lp-2+\theta\lp1-2x\rp
  \rp.
  $$ 
It holds that $F_3$ is decreasing and $F_3(1/2)<0$. Hence, if $F_3(0)>0$   there is a unique zero. This is the case when $r>2.$ 
Finally, consider for $x\geq 0$
$$
F_1(x):=  \frac{p_1}{c_r}\lp g'(x,1)\varphi(x)-\varphi^+(x)g(x,1)\rp=\frac1{3\theta}\, \e^{-\theta x}\lp1+\theta\lp1+x\rp
  \rp.
  $$  
 It is immediately seen that $F_1$ is positive everywhere and decreasing.  
 
We prove now (a). For $r>2$ let $x^*_2\in(0,2)$ be the zero of $F_2$ and  $x^*_3\in(0,1/2)$ the zero of $F_3$. Let $\Sigma_a$ be the set defined by the right hand side of (\ref{S1}) and introduce
 $$
  \tilde H_a:=\inf\{t:\X_t\in\Sigma_a\}
 $$
 and 
  $$
  \tilde V_a(\x):=\mathbb{E}_\x\lp \e^{-r\tilde H_a}g\lp \X_{\tilde H_a}\rp\rp=
  \begin{cases}
  \vspace{1mm}
  g(\x),& \x\in \Sigma_a,\\
  \vspace{1mm}
g(x^*_2,2) {\displaystyle \frac{\varphi(x)}{\varphi(x^*_2)}} ,& x\geq x^*_2\ {\rm on\ leg}\ 2,\\
g(x^*_3,3){\displaystyle \frac{\varphi(x)}{\varphi(x^*_3)}} ,& x\geq x^*_3\ {\rm on\ leg}\ 3.
  \end{cases}
 $$
Recall,  cf.  (\ref{H0}), that for $x\geq y\geq 0$ 
$$
\frac{\varphi(x)}{\varphi(y)}=\mathbb{E}_x\lp \e^{-rH_y}\rp,
$$
where $H_y:=\inf\{t:X_t=y\}.$ Clearly, $\tilde V_a$ dominates $g$. Moreover, it can be checked that $\tilde V_a$ is $r$-excessive by calculating the expressions in (\ref{rep01}) on the legs (in fact, these are the functions $F_1, F_2$, and $F_3$ with some multiplies) and  in (\ref{rep02}). Notice that $\mathbf{D} \tilde V(\0) =\mathbf{D} g(\0)= -\frac 12$ by (\ref{D1}).  Evoking Proposition \ref{apu} completes the proof of (a). 

Consider next the case (b). Analogously as above, we let $\Sigma_b$ denote the right hand side of (\ref{S2}) where $x^*_2$ is the unique zero of $F_2$. Furthermore, 
  $$
  \tilde H_b:=\inf\{t:\X_t\in\Sigma_b\}
 $$
and 
  $$
  \tilde V_b(\x):=\mathbb{E}_\x\lp \e^{-r\tilde H_b}g\lp \X_{\tilde H_b}\rp\rp=
  \begin{cases}
  \vspace{2mm}
  g(\x),& \x\in \Sigma_b,\\
  \vspace{2mm}
g(x^*_2,2) {\displaystyle \frac{\varphi(x)}{\varphi(x^*_2)}} ,& x\geq x^*_2 \ {\rm on\ leg}\ 2,\\
\varphi(x) ,& x\geq 0 \ \ \ {\rm on\ leg}\ 3.
  \end{cases}
 $$
Clearly, $\tilde V_b$ dominates $g$ on the legs 1 and 2.  Since $1/8<r\leq 2$ it is easily seen that $\tilde V_b$ 
dominates $g$ also on the leg 3. We can check the $r$-excessivity of  $\tilde V_b$ by calculating the representing measure. Notice, in particular, that  $\mathbf{D} \tilde V_b(\0) = \frac 1 3 \lp 1-\frac 1 2- \theta\rp\leq 0 $. 

We consider now case (c). 
The three functions $F_1,F_2 $ and $F_3$ above have no roots, so we should consider the function on the l.h.s. of \eqref{rep12} with $i_o=1$.
The function is proportional to
$$
\bar{F}_1(x)=(1+x)\tilde\psi^-(x,1)-\tilde\psi(x,1)
$$
with $\tilde\psi$ in \eqref{tildepsi1}. 
After some computations we obtain
$$
\bar{F}_1(x)={\theta(1+x)\over 2}
\left(
3e^{\theta x}+e^{-\theta x}
\right)-\frac12\left(
3e^{\theta x}-e^{-\theta x}
\right)
$$
For $r\leq 1/8$ we have $\theta\leq 1/2$ and $\bar{F}_1(0)\leq 0$ (with equality only at $r=1/8$).
Furthermore, by differentiation we obtain that the function is increasing and it follows that it has a unique non-negative root $z^*_1$. 
Consider now $\Sigma_c$ the set in the l.h.s. of \eqref{S3} and the stopping time
$$
 \tilde H_c:=\inf\{t:\X_t\in\Sigma_c\}.
 $$
Applying the hitting time in \eqref{lemma1}, we obtain the candidate value function for the problem:
\begin{equation}\label{vece}
\tilde V_c(\x):=\mathbb{E}_\x\lp \e^{-r\tilde H_{z^*_1}}g\lp \X_{\tilde H_{z^*_1}}\rp\rp=
  \begin{cases}
  \vspace{2mm}
  g(\x),& \x\in \Sigma_c,\\
\vspace{2mm}
\displaystyle \frac{g(z^*_1,1)}{\tilde\psi(z^*_1,1)}\tilde\psi(x,1),& 0\leq x\leq z^*_1 \ {\rm on\ leg}\ 1,\\

\displaystyle{g(z^*_1,1)\over \tilde\psi(z^*_1,1)}{\varphi(x)\over\theta} ,& x\geq 0 \  {\rm on\ legs}\ 2,3.
\end{cases}
\end{equation}
We now apply Theorem \ref{apu}. 
The function $\tilde V_c$ is $r$-excessive because its representing measure is positive on the set $(z^*_1,\infty)$ on leg 1 according to the election of 
$z^*_1$ and \eqref{rep12}, and vanishes on the set $(0,z^*_1)$ on leg 1, and also on the other legs since the chosen function in (\ref{vece}) is   harmonic on the complement of $\Sigma_c$.
Let us now prove that 
\begin{equation}
\label{ineq}
\tilde V_c\geq g.
\end{equation}
It is clear that $\tilde V_c(\x)= g(\x)$ for $\x=(x,1)$ and $x\geq z^*_1$.
The condition $\bar F_1(z^*_1)=0$ implies
$$
\tilde\psi'(z^*_1,1)={\tilde\psi(z^*_1,1)\over g(z^*_1,1)},
$$
and then according to \eqref{vece} we have that $\tilde V_c'(z^*_1,1)=1$.
As the function $\tilde\psi$ is strictly convex, it holds
that $\tilde V_c(\x)\geq g(\x)$ everywhere on the leg 1. 
As a consequence, this yields that 
$$
\tilde V_c(\0)={g(z^*_1,1)\over \theta\tilde\psi(z^*_1,1)}=:K>1.
$$
Introduce the tangent line to $\tilde V_c$ at $x=0$ on the leg 2.
As $\tilde V_c(x,2)=Ke^{-\theta x}$, the tangent  is
$$
t(x)=K(1-\theta x).
$$
As $t(2)=K(1-2\theta)\geq 0$ we conclude that $\tilde V_c(x,2)\geq t(x) \geq g(x,2)$. The fact that $g(x,3)<g(x,2)$ concludes the proof of \eqref{ineq}.

\end{proof}

\begin{remark}[OSP for Skew Brownian motion]
In case $n=2$, process $\X$ is the usual skew Brownian motion with parameter $\beta=p_1$, 
identifying the first leg with the positive half-line and the second leg with the negative half-line,
and denoting by $x$ the distance from the origin.
For a general payoff function $g$, considered as a function on the line $\R$, 
	the application of Proposition \ref{COND} gives that when
	$$
	\beta g'(0+)-(1-\beta)g'(0-)>0,
	$$
	the origin $\0$  belongs for all values of the discounting parameter $r$ to the continuation region.
	This result can also be found in \cite{AlvarezSalminen} Proposition 1.
\end{remark}

\subsection{Example}\label{ex2}
Consider next the  optimal stopping problem for a Brownian spider with $n$ legs and payoff function 
\begin{align*}
g(x,i) = A_i x, \ x >0 , \ i = 1,2,...,n,
\end{align*}
where $A_i>0$. Straightforward but a bit lengthy calculations, 
show that the  representation (\ref{ma3}) for $g$ holds  (cf. Proposition \ref{rwdmeasure}) when 
\begin{align*}
\Delta_\0=-\mathbf{D}g(\mathbf{\0}) = -\sum_{i=1}^n p_iA_i
\end{align*}
and for $x>0$, see (\ref{msx}), 
\begin{equation*}
	f(x,i)=- \lp \frac 12 \frac{d^2}{dx^2} (A_i x)- rA_i x\rp = r\, A_i x.
	\end{equation*}
 Applying  Proposition \ref{fin00}, with the harmonic function given in (\ref{mahar23}), 
  it is seen that the value function $V$ of the problem is finite. From Theorem \ref{st1} we know that $V$ is the smallest $r$-excessive majorant of $g$ and the stopping set $\Sigma$ is given by (\ref{stop1}). Since $g$ and $V$ are continuous  it follows that $\Sigma$ is closed and, in particular,  $V(z,i)=g(z,i)$  when $(z,i)\in\partial\Sigma$. Moreover, notice  that $\0\not\in\Sigma$ because  $g(\0)=0$ (cf. also Proposition \ref{COND}). Recall that $V$ is $r$-harmonic on $C:=\Sigma^c$ and $V=g$ on $\Sigma$. Hence $V$ can be represented for all $(x,i)\in\Gamma$ as 
\begin{equation}
\label{rie2}
 V(x,i)=\sum_{k=1}^n \int_{0}^{+\infty}\g_r((x,i),(y,k))\,{\bf 1}_\Sigma((y,k))\,\,f(y,k)\, dy,
\end{equation}
 where the Green function $\g_r$ is given in \eqref{kernel}. We remark that (\ref{rie2}) is the Riesz representation of the $r$-excessive function $V$. Here it is also used that, due to the smooth pasting property (valid in this case, see Theorem\ref{smoothfit1} in Appendix B), the representing measure of $V$ does not put mass on the boundary points of $\Sigma$, see \eqref{rep03}, and is, therefore, absolutely continuous with respect to the Lebesgue measure. The density is $f$ since $V=g$ on $\Sigma$. Consequently, it holds (cf. (\ref{b0})) for all $(z,i)\in\partial \Sigma$ 
\begin{align}\label{b0x}
\sum_{k=1}^n \int_{0}^{+\infty}\g_r((z,i),(y,k))\,{\bf 1}_{\Sigma^c}((y,k))\, {2rp_k}\, A_k y\, dy= \g_r((z,i),\0)\sum_{k=1}^n p_k A_k. 
\end{align}
\begin{proposition}
\label{cone} There exist $z^*_i>0,\ i=1,2,...,n,$ such that  $\Sigma=\{ (x,i)\,:\, x\geq z^*_i, i=1,2,...,n\}$.
\end{proposition}
\relax
\begin{proof} 
Let
$$
A_*=\min\{A_1,...,A_n\}\quad {\rm and}\quad  A^*=\max\{A_1,...,A_n\}.
$$
Let $V_*$ ($V^*$) be the value function of OSP with reward $g_*(x,i)=A_*x$ 
($g_*(x,i)=A^*x$), $i=1,2,...,n$. Then, because for all $x$ and $i$
$$
g_*(x,i)\leq g(x,i)\leq g^*(x,i),
$$
it follows 
\begin{equation}\label{bound0}
V_*(x,i)\leq V(x,i)\leq V^*(x,i).
\end{equation}
By symmetry, for all $x$ and $i$ 
\begin{equation}\label{b0x1}
V_*(\x)=\sup_{\tau}\EE_{\x}\left(e^{-r\tau}\,A_*\,\X(\tau)\right)
=A_*\sup_{\eta}\EE_{x}\left(e^{-r\eta}\, X_\eta\right)
\end{equation}
and
\begin{equation}\label{b0x2}
V^*(\x)=\sup_{\tau}\EE_{\x}\left(e^{-r\tau}\,A^*\,\X(\tau)\right)
=A^*\sup_{\eta}\EE_{x}\left(e^{-r\eta}\, X_\eta\right),
\end{equation}
where $X=(X_t)_{t\geq 0}$ denotes a reflecting Brownian motion, $\EE_{x}$ the expectation operator associated with $X$ when started at $x$ and the supremum in case of $X$ is taken over all stopping times $\eta$ in the filtration of $X$. For OSP (\ref{b0x1}) and for (\ref{b0x2}) there exists $z_*$ and $z^*$, respectively, such that the corresponding stopping regions are given by $\Sigma_*=\{ (x,i)\,:\, x\geq z_*, i=1,2,...,n\}$ and $\Sigma^* =\{ (x,i)\,:\, x\geq z^*, i=1,2,...,n\}$. In particular, on every leg $i$ it holds for $x\geq z^*$
\begin{equation}\label{bound}
V(x,i)\leq A^* x.
\end{equation}
Consequently, $\Sigma^c$ does not contain sets of the form $O_{i,z}:=\{(x,i)\,:\, x>z\}$ for any $i$ and $z>0$.  Indeed, if there is $i$ and $z$ such  $O_{i,z}\subset\Sigma^c$ then $V$ is $r$-harmonic on $O_{i,z}$ and increases exponentially, see (\ref{mahar23}), but this violates (\ref{bound}). Next, (\ref{bound0}) also yields that $\Sigma\subseteq \{(x,i)\,:\, x\geq z_*, i=1,2,...,n\}$. Finally, assume that $\Sigma$  is not connected on some leg $k$  Then there exist $x_k\in\Sigma^c$ such that $V(x_k,k)>g(x_k,k)=A_kx_k$, and, by continuity, $x_{k_1}$ and $x_{k_2}$ in $\Sigma$ such that $x_{k_1}<x_k<x_{k_2}$ and $V(x,k)>A_k x$ for all in $x\in(x_{k_1},x_{k_2})$.  Let 
$$
H_{1,2}:=\inf\{t\,:\, \X_t\notin (x_{k_1},x_{k_2})\}.
$$
Since $V$ is $r$-harmonic on $(x_{k_1},x_{k_2})$ it holds
$$
V(x,k)=\EE_{(x,k)}\left(e^{-rH_{1,2}}\,A_k\,\X(H_{1,2})\right)=A_k\EE_{x}\left(e^{-rH_{1,2}}\,X_{H_{1,2}}\right).
$$
Using  \cite{BorodinSalminen} 3.3.0.5 (a) and (b) p. 218 gives 
\begin{align*}
\EE_{x}\left(e^{-rH_{1,2}}\,X_{H_{1,2}}\right)&=x_{k_1}\EE_{x}\left(e^{-rH_{1,2}}\,;\,X_{H_{1,2}}=x_{k_1}\right)+x_{k_2}\EE_{x}\left(e^{-rH_{1,2}}\,;\,X_{H_{1,2}}=x_{k_2}\right)\cr
&=\frac 1 {\sinh\lp(x_{k_2}-x_{k_1})\sqrt{2r}\rp}\lp x_{k_1}\sinh\lp(x_{k_2}-x)\sqrt{2r}\rp
+x_{k_2}\sinh\lp(x-x_{k_1})\sqrt{2r}\rp\rp.
\end{align*}
Consequently, $V$ is convex on $(x_{k_1},x_{k_2})$ and, hence, $V(x,k)<g(x,k)$ for all $x\in(x_{k_1},x_{k_2})$. This is a contradiction, and the proof is complete.
\end{proof}
Let $z^*_i, i=1,2,...,n,$ be the boundary points of $\Sigma$, as introduced in Proposition \ref{cone}. Then $(z^*_1,..., z^*_n)$ constitutes a solution to the system of the equations (\ref{b0x}), i.e.,  for $i=1,2,...,n$ 
 \begin{align}\label{b0xx}
\sum_{k=1}^n \int_{0}^{z^*_k}\g_r((z^*_i,i),(y,k))\, {2rp_k}\, A_k y\, dy=\g_r((z^*_i,i),\0)\sum_{k=1}^n p_k A_k. 
\end{align}
The uniqueness of the solution  follows from Theorem \ref{uniqueness_green}. Indeed, it is easily verifed that letting $\cal U$  consist of  all sets $U$ of the form
 $$
U=U(\alpha_1,...,\alpha_n)=\{(x_i,i)\,:\, x_i\geq \alpha_i, i=1,...,n\}, \quad \alpha_i>0, i=1,2,...,n,
$$
then the properties {\bf (a)}, {\bf(b)}, {\bf (c)}, and  {\bf (d)} listed therein hold.

To give some numerical flavor, let $n=3$ and $r=1/2$. Then the system (\ref{b0xx}) becomes
\begin{align*}
A_i \int_0^{z_i} y\sinh(y)dy + \sum_{k=1}^3 A_k p_k \int_0^{z_k} y\e^{-y} dy = \sum_{k=1}^3 A_k p_k, \quad i=1,2,3,
\end{align*}
and integrating yields
\begin{align}
\label{se}
A_i(z_i \cosh(z_i) - \sinh(z_i))			= \sum_{k=1}^3 A_k p_k \e^{-z_k}(1+z_k), \quad i=1,2,3. 
\end{align}
Putting, moreover, $p_1=p_2=p_3=1/3$, $A_1=1, A_2=2$ and $A_3=3$ and  solving numerically equations (\ref{se}) we obtain the stopping region of the OSP
that has the respective critical threshold values
$$
z_1=1.4816,\quad z_2=1.2041,\quad z_3=1.0628.
$$

\subsection{Example}
\label{ex3}
We conclude by considering in the same setting  as in Example \ref{ex2} the  optimal stopping problem  with the payoff function 
\begin{align*}
g(x,i) = A_i x^2, \ x >0 , \ i = 1,2,3,
\end{align*}
where $A_i>0$. Again, applying  Proposition \ref{fin00} it is seen that the value function of the problem is finite. Straightforward but a bit lengthy calculations show the the  representation (\ref{ma3}) for $g$ holds  (cf. Proposition \ref{rwdmeasure}) when
\begin{align*}
\Delta_\0=-\mathbf{D}g(\mathbf{\0}) = 0
\end{align*}
and for $x>0$, see (\ref{msx}), 
\begin{equation*}
	f(x,i)=- \lp \frac 12 \frac{d^2}{dx^2} (A_i x^2)- rA_i x^2\rp 
= A _i\lp rx^2-1\rp.
\end{equation*}
The corresponding equations obtained from \eqref{b0} then are:
\begin{align}\label{b0xx}
\sum_{k=1}^n \int_{0}^{+\infty}\g_r((z,i),(y,k))\,{\bf 1}_{\Sigma^c}((y,k))\, A_k(ry^2-1)2p_k\, dy= 0. 
\end{align}
For simplicity we take $r=1/2$. 
We use the Green kernel expression \eqref{kernel} and proceed similarly as in Example \ref{ex2}, (in particular, it is seen that the continuation region is connected, and that the solution of the associated equation  system is unique).
Introducing the auxiliary functions
\begin{align*}
F(z)&=\int_0^z\sinh(y)(y^2-1)dy=z^2\cosh(z)-2z\sinh(z)\\
H(z)&=\int_0^z e^{-y}(y^2-1)dy=-e^{-z}z(z+2).
\end{align*}
the equation system to find the boundary points of the continuation region is as follows
$$
A_1 F(z^*_1)=A_2 F(z^*_2)=A_3F(z^*_3)
$$
and
$$
A_1 \big(F(z^*_1)+H(z^*_1)\big)+A_2 \big(F(z^*_2)+H(z^*_2)\big)+A_3\big(F(z^*_3)+H(z^*_3)\big)=0.
$$
To solve numerically, we take $A_1=1$, $A_2=2$ and $A_3=3$ obtaining
$$
z^*_1= 2.16987, \qquad z^*_2=2.06543, \qquad z^*_3=2.02250.
$$

\begin{appendices}
\section{Proof of Proposition \ref{rwdmeasure}} 
We prove first (a). For this, consider for $x\geq 0$ 
\begin{align*}
    g(x,i) &= \int_{0}^{+\infty}\g_r((x,i),(y,i))\,f(y,i)\, p_i m(dy)\\
&\hskip3cm+\sum_{k=1}^n \int_{0}^{+\infty}\g_r((x,i),(y,k))\,f(y,k)\, p_k m(dy)  + 
    \g_r((x,i),\0)\,\Delta_\0\cr
&= \int_{(0,x]}\varphi(x)\tilde \psi(y,i)\,f(y,i)\, p_i m(dy)+\int_{(x,\infty)}\varphi(y)\tilde \psi(x,i)\,f(y,i)\, p_i m(dy) \cr
&\hskip2cm
+\frac 1{c_r}\sum_{k\not=i} \int_{0}^{+\infty}\varphi(x)\varphi(y)\,f(y,k)\, p_k m(dy)  + \varphi(x) \tilde \psi(\0) \,\Delta_\0,
\end{align*}
where the explicit form of the Green kernel is used, in particular, $\tilde \psi(y,i)=p_i^{-1}\psi^\partial(y)+c_r^{-1}\varphi(y)$ and 
$\tilde \psi(\0)=c_r^{-1}$ (since $\psi^\partial(0)=0$ and $\varphi(0)=1)$. Introduce a signed measure on  $\Gamma$ via   
\begin{align*}
\mu(dy,i):=&\, \varphi(y)f(y,i)\,p_i m(dy),\quad  y>0, i=1,2,...,n,\cr
\mu(\0):= &\, \Delta_\0.
\end{align*}
Notice that, by assumption (\ref{ma3011}), $\mu$ is of finite total variation.  Straightforward manipulations (cf. (\ref{fu1}) and (\ref{v1})) yield 
\begin{align}
\label{ma301}
    g(x,i) = \frac 1{p_i}\varphi(x) u(x,i) + \frac 1{c_r}\varphi(x)\mu(\Sigma),
  \end{align}  
where 
\begin{align}
\label{ma302}
    u(x,i) = \int_{(0,x]}\frac{ \psi^\partial(y)}{\varphi(y)}\,\mu(dy,i)+ \psi^\partial (x)\,\mu((x,\infty),i). 
  \end{align} 
Next we calculate the scale derivative of $g$ at $(0,i)$. Since $\varphi^+(0)=-c_r$ and  $u(0,i)=0$  we have from (\ref{ma301}) 
$$
g^+(0,i)= \frac 1{p_i}u^+(0,i) -\mu(\Gamma),
$$
and $g^+(0,i)$ exists if and only if   $u^+(0,i)$  exists. Hence consider
$$
\lim_{\delta\downarrow 0}\frac{u(y,i)-u(0,i)}{S(y)-S(0)}=\lim_{\delta\downarrow 0}\frac{u(\delta,i)}{S(\delta)}.
 $$   
 The scale derivative of the first term on the r.h.s. of (\ref{ma302}) is zero. To show this, introduce for $y>0$ and $i=1,2,...,n$
 $$
 \mu_+(dy,i):=\varphi(y)f(y,i){\bf 1}_{\{f(y,i)\geq 0\}}p_i m(dy)\quad  {\rm { and }}\quad  
 \mu_-(dy,i):=-\varphi(y)f(y,i){\bf 1}_{\{f(y,i)< 0\}}p_i m(dy),
 $$
 and then 
$$
\lim_{\delta\downarrow 0}\Big|\frac 1{S(\delta) }\int_{(0,\delta]}\frac{ \psi^\partial(y)}
{\varphi(y)}\,\mu(dy,i)\Big| \leq 
\lim_{\delta\downarrow 0}\lp\frac 1{S(\delta) }\frac{ \psi^\partial(\delta)}
{\varphi(\delta)}\,\big(\mu_+((0,\delta],i)+\mu_-((0,\delta],i)\big)\rp=0,
$$
since $\mu_+$ and $\mu_-$ are finite measures and $\lim_{\delta\downarrow 0}\psi^\partial(\delta)/S(\delta)= 1$ (see (\ref{norm})). Consequently, 
$$
u^+(0,i)=\lim_{\delta\downarrow 0}\frac 1{S(\delta) }\mu((\delta,\infty),i)=\mu((0,\infty),i),
$$
and, hence,
\begin{align*}
{\bf D}g(\0)&=\sum_{i=1}^n p_i\,g^+(0,i)= \sum_{i=1}^n p_i\Big(\frac 1{p_i}u^+(0,i) -\mu(\Gamma)\Big)=\sum_{i=1}^n\mu((0,\infty),i)  -\mu(\Gamma)\cr
&=-\mu(\0)=-\Delta_0
\end{align*}
proving (a). 

Consider now the claim (b).  To deduce formula (\ref{ms}) the scale derivative of $g(\cdot,i)$ is needed. From  (\ref{ma301}) it is seen then that we should calculate the scale derivative of $u(\cdot,i)$. It holds for $\delta\geq 0$
\begin{align*}
{u(x+\delta,i)-u(x,i)}& = \int_{(x,x+\delta]}\frac{\psi^\partial(y)}{\varphi(y)}\,  \mu(dy,i)+\frac{\psi^\partial(x+\delta)}{\varphi(x+\delta)}\mu((x+\delta,\infty),i) -\frac{\psi^\partial(x)}{\varphi(x)}\mu((x,\infty),i)\cr
&=  \int_{(x,x+\delta]}\frac{\psi^\partial(y)}{\varphi(y)}\,  \mu(dy,i)-\frac{\psi^\partial(x)}{\varphi(x)}\mu((x,x+\delta],i)\cr
&\hskip3cm  +\lp\frac{\psi^\partial(x+\delta)}{\varphi(x+\delta)}- \frac{\psi^\partial(x)}{\varphi(x)}\rp \mu((x+\delta,\infty),i)\cr
&=  \int_{(x,x+\delta]}\lp \frac{\psi^\partial(y)}{\varphi(y)}-\frac{\psi^\partial(x)}{\varphi(x)}\rp \,  \mu(dy,i)\cr
&\hskip3cm  +\lp\frac{\psi^\partial(x+\delta)}{\varphi(x+\delta)}- \frac{\psi^\partial(x)}{\varphi(x)}\rp \mu((x+\delta,\infty),i)\cr
&=:F_1(x,\delta)+F_2(x,\delta).
\end{align*}
Since $x\mapsto \psi^\partial(x)/\varphi(x)$ is non-decreasing and (cf. (\ref{w0}))
$$
\lim_{\delta \downarrow 0} \frac{1}{S(x+\delta)-S(x)}\lp \frac{\psi^\partial(x+\delta)}{\varphi(x+\delta)}- \frac{\psi^\partial(x)}{\varphi(x)}\rp =\frac 1{\varphi^2(x)}
$$
we have 
$$
\lim_{\delta \downarrow 0} \frac{F_1(x,\delta)}{S(x+\delta)-S(x)}=0
$$
and 
$$
\lim_{\delta \downarrow 0} \frac{F_2(x,\delta)}{S(x+\delta)-S(x)}=\frac {\mu((x,\infty),i)}{\varphi^2(x)}.
$$
Straightforward manipulations yield
$$
g^+(x,i)\varphi(x)  - g(x,i)\varphi^+(x)=\frac 1{p_i}\mu((x,\infty),i),
$$
as claimed in (b).

\section{Smooth fit theorem for diffusion spiders} 

To make the paper more self contained we discuss here the smooth fit property for diffusion spiders. In the formulation below the differentiation is taken with respect to an arbitrary continuous and increasing function $F$; the "usual" choices being $F=S$ (the scale function) or $F(x)=x, x>0.$ Notice that we consider the smooth fit "only" on the legs of the spider (and not in the vertex). In fact, in this case, the result is as in the diffusion case presented in Salminen and Ta 
\cite{SalminenTa} (and earlier in \cite{Salminen}), see also Peskir \cite{Peskir07}.

\begin{theorem}
\label{smoothfit1}
Let $z\not= \0$ be a left boundary point of the stopping set $\Sigma ,$ i.e., 
$[z,z+\varepsilon_1)\subset \Sigma\setminus \{\0\}$ and $(z-\varepsilon_2,z)\subset
  \Sigma^{\,c}\setminus \{\0\}$ for some positive $\varepsilon_1$ and  $\varepsilon_2.$
Let $F$ be a continuous and increasing function and assume that the
reward function 
$g$ and the functions $\varphi_\alpha$
and $\tilde\psi_\alpha,\ \alpha\geq 0,$ are  $F$-differentiable at $z.$ Then the value
function $V$ in (\ref{eq:osp})  is $F$-differentiable at $z$ and the smooth fit with respect
to $F$   holds:
\begin{equation}
\label{sfeq}
\frac{d^+V}{dF}(z)= \frac{d^-V}{dF}(z)=\frac{dg}{dF}(z).
\end{equation}
\end{theorem}
\begin{proof}
Since $V>g$ on $\Sigma^c$ and $V=g$ on $\Sigma$ 
we have 
\begin{align*}
\frac{d^+V}{dF}(z)&=\lim_{\delta\to 0+}
\frac{V(z+\delta)-V(z)}{F(z+\delta)-F(z)}
=\lim_{\delta\to 0+}
\frac{g(z+\delta)-g(z)}{F(z+\delta)-F(z)}
\\
&=\frac{d^+g}{dF}(z)
=\frac{dg}{dF}(z)
\end{align*}
and
\begin{align*}
\frac{d^-V}{dF}(z)&=\lim_{\delta\to 0+}
\frac{V(z-\delta)-V(z)}{F(z-\delta)-F(z)}
=\lim_{\delta\to 0+}
\frac{V(z)-V(z-\delta)}{F(z)-F(z-\delta)}
\\
&\leq
\lim_{\delta\to 0+}
\frac{g(z)-g(z-\delta)}{F(z)-F(z-\delta)}
\\
&=
\frac{d^-g}{dF}(z)
=\frac{dg}{dF}(z).
\end{align*}
Consequently,
$$
\frac{d^-V}{dF}(z)\leq \frac{d^+V}{dF}(z), 
$$
and, hence, (\ref{glue2}) implies
$$
\frac{d^-V}{dF}(z)= \frac{d^+V}{dF}(z),
$$
proving the claim.
\end{proof}
\end{appendices}

\subsubsection*{Acknowledgements} 
This research has been partially supported by grants from Magnus Ehrnrooths stiftelse, Finland. 
The authors acknowledge also the hospitality of \AA bo Akademi University, Turku-Åbo, Finland and Universidad de la República, Montevideo, Uruguay.


\begin{thebibliography}{}


\bibitem{AlvarezSalminen}
Alvarez, L.H.R. and  Salminen, P.:
Timing in the presence of directional predictability: optimal stopping of skew Brownian motion.
 \newblock {\em  Mathematical Methods of Operations Research}, Vol. 86, pp. 377--400, 2017.
 
 \bibitem{BarlowPitmanYor1}
Barlow, M.T., Pitman, J. and Yor, M.:
On Walsh's Brownian motions.
{\em S\'eminaire de probabilit\'es de Strasbourg}, 
Vol. 23, pp. 275--293, 1989.



\bibitem{BarlowPitmanYor}
Barlow, M.T., Pitman, J. and Yor, M.:
Une extension multidimensionnelle de la loi de l'arc sinus.
{\em S\'eminaire de probabilit\'es de Strasbourg}, 
Vol. 23, pp. 294--314, 1989.



\bibitem{BayZhang} 
Bayraktar, E. and Zhang, X.: Embedding of Walsh Brownian Motion.  \newblock {\em  Stochastic Processes and their Applications}, Vol 134, pp. 1--28, 2021.  

\bibitem{BorodinSalminen}
Borodin, A.N.  and Salminen, P.:
\newblock {\em Handbook of {B}rownian motion---Facts and formulae},
2nd. ed., 2nd corr. print.,  
Birkh{\"a}user, Springer Basel AG,
2015.



 
 \bibitem{CCMS}
 Christensen, S., Crocce, F., Mordecki, E. and Salminen, P.:
On optimal stopping of multidimensional diffusions.
 \newblock {\em Stochastic Processes and their Applications}, vol 129 , pp. 2561--2581, 2019.

 \bibitem{CCFP1}
 Cs\'aki, E., Cs\"orgo, M., F\"oldes, A. and R\'ev\'esz, P.: Some limit theorems for heights of random walks on a spider. \newblock {\em Journal of Theoretical Probability}, Vol. 29, pp. 1685--1709, 2016.
  
\bibitem{CCFP2}
Cs\'aki, E., Cs\"orgo, M., F\"oldes, A. and R\'ev\'esz, P.: Limit theorems for local and occupation times of random walks and Brownian motion on a spider. \newblock {\em Journal of Theoretical probability}, Vol. 32, pp. 330--352, 2019.



\bibitem{CM14}
Crocce, F. and  Mordecki, E.:
Explicit solutions in one-sided optimal stopping problems for one-dimensional diffusions.
 \newblock {\em Stochastics: An International Journal of Probability and Stochastic Processes}, Vol  86(3), pp. 491--509, 2014. 


\bibitem{CrocceMordecki}
Crocce, F. and Mordecki, E.:
An algorithm to solve optimal stopping problems for one-dimensional diffusions. arXiv:1909.10257, (to appear in {\em ALEA}).

\bibitem{DassiosZhang}
Dassios, M. and  Zhang, J.:
First Hitting Time of Brownian Motion on Simple Graph with Skew Semiaxes. {\em  Methodology and Computing in Applied Probability}, Vol. 24, pp.1805–-1831, 2022.

\bibitem{duToitPeskir}
du Toit, J. and  Peskir, G.:  Selling a stock at the ultimate maximum. {\em  Annals of Applied Probability}, Vol. 19 (3), pp. 983--1014, 2009.  

\bibitem{Ernst}
Ernst P.,
Exercising control when confronted by a (Brownian) spider,
{\em Operations Research Letters},
Vol 44 (4),
pp. 487--490,
2016.


\bibitem{FitzsimmonsKuter}
Fitzsimmons, P. J. and Kuter, K. E.:
Harmonic functions on Walsh's Brownian motion.
 \newblock {\em Stochastic Processes and their Applications}, Vol 124,  pp. 2228--2248, 2014.

\bibitem{fk}
Fitzsimmons, P. J. and  Kuter, K. E.: Harmonic functions of Brownian motions on metric graphs.
 \newblock {\em Journal of Mathematical Physics}, Vol  56, 013504, 2015.



\bibitem{fs}
Freidlin, M.I. and Sheu, S-J.:
Diffusion processes on graphs: stochastic differential equations, large deviation principle. \newblock {\em 
Probability Theory and  Related Fields}, Vol 116, pp.181--220, 2000. 

\bibitem{FreidlinWentzell}
 Freidlin, M.I. and  Wentzell, A.D.: 
Diffusion processes on graphs and the averaging principle. 
 \newblock {\em Annals of Probability}, Vol  21(4), pp. 2215--2245, 1993.


\bibitem{KaratzasYan} 
Karatzas, I. and Yan, M.:  Semimartingales on rays, Walsh diffusions, and related problems of control and stopping. 
\newblock {\em Stochastic Processes and their Applications}, Vol 129, pp. 1921--1963, 2019.

\bibitem{Kostrykin}
Kostrykin, V.,  Potthoff, J. and  Schrader, R.: 
Brownian motions on metric graphs.
Journal of Mathematical. Physics, Vol. 53(9), 2012.

\bibitem{KunitaWatanabe} 
Kunita, H.  and Watanabe, T.:   Markov processes and Martin boundaries, part I. 
\newblock {\em Illinois Journal of Mathematics}, Vol  9, pp. 485-526, 1965.

\bibitem{MordeckiSalminen1}
Mordecki, E. and Salminen, P.: 
Optimal stopping of Brownian motion with broken drift.
 \newblock {\em High Frequency}, Vol  2, pp. 113-120, 2019.

\bibitem{MordeckiSalminen2}
Mordecki, E. and Salminen, P.: 
Optimal stopping of oscillating Brownian motion.  \newblock {\em Electronic Communications in Probability}, 
Vol  24(50), pp. 1-12, 2019.

\bibitem{PapaPapaLep}  
Papanicolaou, V.G., Papageorgiou, E.G. and Lepipas, D.C.:
Random motion on simple graph.
 \newblock 
 {\em  Methodology and Computing in Applied Probability},
 Vol 14, pp. 285--297, 2012.

\bibitem{Peskir}
Peskir, G.:  On the American option problem. {\em  Mathematical Finance}, Vol. 15(1), pp. 169--181, 2005.

\bibitem{Peskir07}
Peskir, G.: 
Principle of smooth fit and diffusions with angles. 
{\em Stochastics}, Vol 79, pp. 293--302, 2007.

\bibitem{RevuzYor}
Revuz, D. and Yor, M.:
\newblock \emph{Continuous martingales and {B}rownian motion, 3rd ed.}, 
\newblock Springer-Verlag, Berlin, 1999.


\bibitem{Rogers}
Rogers, L.C.G.:
It\^o excursion theory via resolvents.
\newblock {\em Z. Wahrscheinlichkeitstheorie verw. Gebiete}, 
Vol 63, pp. 237--255, 1983.

\bibitem{Salisbury}
Salisbury, T.S.: 
Construction of right processes from excursions. 
{\em  Probababiliy Theory and Related Fields}, 
Vol. 73, pp.351--367, 1986.


\bibitem{Salminen}
Salminen, P.:
Optimal stopping of one-dimensional diffusions
 \newblock {\em Mathematische Nachrichten}, 
 Vol 124, pp. 85--101, 1985.
 
 \bibitem{SalminenTa}
Salminen, P. and Ta, B.Q.:
 Differentiability of excessive functions of
one-dimensional diffusions and the principle of smooth fit.
 \newblock {\em Banach Center Publications}, 
 Vol. 104,
pp. 181--199, 2015.
 

 
 \bibitem{salminenvalloisyor07}
Salminen, P., Vallois, P. and Yor, M.:
On the excursion theory for linear diffusions.
 \newblock {\em Japanese Journal of Mathematics}, Vol  2, pp. 97--127, 2000.

\bibitem{Shiryaev}
Shiryaev, A.N.:
\newblock {\em Optimal Stopping Rules}.
\newblock  Springer-Verlag, New York, 1978.

\bibitem{Walsh}
Walsh, J.B.: 
A diffusion with a discontinuous local time. In: Temps Locaux,
 \newblock {\em Ast\'erisque}, Vol  52--53, pp. 37--45, 1978.




\end{thebibliography}
\end{document}